\documentclass[12pt]{article}

\usepackage{amsfonts, amsmath, hyperref, latexsym, amssymb, amsthm, url,tikz, xcolor, rotating, array, bbm,}
\usepackage{rotating}
\usepackage[margin=2cm]{caption}

\newtheorem{theorem}{Theorem}[section]
\newtheorem{lemma}[theorem]{Lemma}
\newtheorem{corollary}[theorem]{Corollary}

\newtheorem{proposition} [theorem]{Proposition}

\newtheorem{remark}[theorem]{Remark}

\theoremstyle{definition}
\newtheorem{definition}{Definition}[section]

\title{Normal limit laws for vertex degrees in randomly grown hooking networks and bipolar networks}
\author{
Colin Desmarais\thanks{Supported by the Swedish Research Council, the Knut and Alice Wallenberg Foundation, and the Swedish Foundation's starting grant from the Ragnar S\"oderberg Foundation} 
\hspace{15mm}
Cecilia Holmgren\footnotemark[1] \\
\small{Department of Mathematics}\\
\small{Uppsala University}\\
\small{Uppsala, Sweden}\\
\texttt{\small\{colin.desmarais, cecilia.holmgren\}@math.uu.se} 
}

\begin{document}

\maketitle

\abstract{ We consider two types of random networks grown in blocks. Hooking networks are grown from a set of 
graphs as blocks, each with a labelled vertex called a hook. At each step in the growth of the network, a vertex called a latch is chosen from the hooking network and a copy of one of the blocks is attached by fusing its hook with the latch. Bipolar networks are grown from a set of directed graphs as blocks, each with a single source and a single sink. At each step in the growth of the network, an arc is chosen and is replaced with a copy of one of the blocks. Using P\'olya urns, we prove normal limit laws for the degree distributions of both networks. We extend previous results by allowing for more than one block in the growth of the networks and by studying arbitrarily large degrees.
}



\section{Introduction}\label{sec:intro}

Several random tree models have been studied where at each step in the growth of the network, a vertex $v$ is chosen amongst all the vertices of the tree, and a child is added to $v$. When the choice of $v$ is made uniformly at random, these trees are called {\em random recursive trees}. When the choice of $v$ is made proportionally to its degree $\deg(v)$, these trees are called {\em random plane-oriented recursive trees}. 
Both models are examples of {\em preferential attachment trees}, where the choice of $v$ is made proportionally to $\chi \deg(v) + \rho$ for real parameters $\chi$ and $\rho$ (notice that a preferential attachment tree is a random recursive tree when $\chi = 0$ and is a random plane-oriented recursive tree when $\rho = 0$). P\'olya urns were used to prove multivariate normal limit laws for the degree distributions in all of these random tree models \cite{MASM:92, MASS:93, JANS:05,HOJS:17}. Asymptotic normality of degree sequences of similar types of preferential attachment models have also been established without the use of P\'olya urns \cite{RESA:16,WARE:17}.

The process of adding a child to a vertex $v$ in a tree can instead be thought of as taking the graph $K_2$ (two vertices joined by an edge) with one of the vertices labelled $h$, and fusing together the vertices $v$ and $h$. Hooking networks are grown in a similar manner from a set of graphs $\mathcal{C} = \{G_1, G_2, \ldots, G_m\}$, called {\em blocks}, where each block $G_i$ has a labelled vertex $h_i$ called a {\em hook}. At each step in the growth of the network, a vertex $v$ called a {\em latch} is chosen from the network, a block $G_i$ is chosen, and the hook $h_i$ and the vertex $v$ are fused together. A more precise formulation is laid out in Section \ref{sec:hooking}.

Several graphs can be thought of as hooking networks. Any tree can be grown as a hooking network with $K_2$ as the only block. A block graph (or clique graph) is a hooking network whose blocks are complete graphs, and a cactus graph is a hooking network whose blocks are cycles.

We prove multivariate normal limit laws for the degree distributions of hooking networks as the number of blocks attached tends to infinity (see Theorem \ref{thm:hooking}). We allow for a preferential attachment scheme for the choice of the latch (i.e., the latch $v$ is chosen proportionally to $\chi \deg(v) + \rho$). We also assign to each block $G_i$ a value $p_i$ such that $p_1 + p_2 + \cdots + p_m = 1$, and choose the block $G_i$  to be attached with probability $p_i$. 

Along with the results for degree distributions of the random tree models described above, Theorem \ref{thm:hooking} also generalizes other results on previously studied hooking networks. Gopaladesikan, Mahmoud, and Ward \cite{GOMW:14} introduced {\em blocks trees}, which can be thought of as hooking networks grown from a set of trees as blocks, where the root of each block has a single child and acts as the hook. In their model, the latch is chosen uniformly at random at each step, and the block to be attached is chosen according to an assigned probability value. They proved a normal limit law for the number of leaves (vertices with degree 1) in blocks trees. Mahmoud \cite{MAHM:19} proved multivariate normal limit laws for the number of vertices with small degrees in {\em self-similar hooking networks}, which are hooking networks grown from a single block called a {\em seed}. Both the case where the latch is chosen uniformly at random and the case where the latch is chosen proportionally to its degree were studied in \cite{MAHM:19}. 
In the extended abstract \cite{DEHO:19}, we presented a proof of multivariate normal limit laws in the specific cases of hooking networks grown from several blocks when the choice of the latch as well as the choice of the block to be attached are made uniformly at random


The methods used to prove our results for hooking networks also apply to proving multivariate normal limit laws for outdegree distributions of bipolar networks (see Theorem \ref{thm:bipolar}). Bipolar networks are grown from a set $\mathcal{C} = \{B_1, B_2, \ldots, B_m\}$ of directed graphs, each with a single source $N_i$: a vertex with zero indegree ($\deg^-(N_i) = 0$), and a single sink $S_i$: a vertex with zero outdegree ($\deg^+(S_i) = 0$). At each step in the growth of the network, an arc $(v,u)$ is chosen and is replaced with one of the blocks $B_i$, by fusing $N_i$ to $v$ and $S_i$ to $u$; see Section \ref{sec:bipolar} for a more precise description. Previously, results were obtained for vertices of small outdegrees in bipolar networks grown from a single block, and where the arc $(v,u)$ to be replaced is chosen uniformly at random \cite{CHMA:16}. We extend previous results by looking at bipolar networks grown from more than one block, by generalizing the choice of the arc to be replaced, and by studying arbitrarily large degrees.

\subsection{Composition of the paper}\label{sec:composition}
 
 The networks studied are described in more detail in Section \ref{sec:networks}. Alongside the descriptions of the networks, running examples of hooking networks and bipolar networks are described in Sections \ref{sec:hooking} and \ref{sec:bipolar} respectively. Our main results are stated in Section \ref{sec:main}. These include multivariate normal limit laws for the vectors of degrees of hooking networks and vectors of outdegrees of bipolar networks. 
 
The theory of generalized P\'olya urns developed by Janson in \cite{JANS:04}, which is the main tool used in the proofs, is summarized in Section \ref{sec:polya}. 

The proofs of our main results are presented in Section \ref{sec:proofs}. This is done in three steps. We start by describing how we study the vertices in our networks as balls in urns in Section \ref{sec:vertsasballs}. Properties of the intensity matrices for these urns are gathered in Section \ref{sec:A}. In Section \ref{sec:mainproofs}, we prove that the matrices studied in \ref{sec:A} are indeed the intensity matrices for the urns we are studying and, with the help of theorems proved in \cite{JANS:04} and stated in Section \ref{sec:polya}, we finish the proofs of our main results.

\subsection{The networks studied}\label{sec:networks}

In the growth of hooking networks and in the growth of bipolar networks, a vertex $v$ is chosen at every step. The choice of the vertex $v$ is made with probability proportional to $\chi \deg(v) + \rho$ in the case of the hooking networks and proportional to $\chi \deg^+(v) + \rho$ in the case of the bipolar networks, where $\chi \geq 0$ and $\rho \in \mathbb{R}$ so that $\chi + \rho > 0$. Since these choices are made proportionally, without loss of generality, we can limit the choice of $\chi$ to $0$ or $1$ (simply divide the numerators and denominators of (\ref{eq:hooklatchprob}) and (\ref{eq:biplatchprob}) below by $\chi$ if this value is nonzero). When $\chi = 1$ we let $\rho > -1$, while we let $\rho$ be strictly positive when $\chi =0$ to avoid the cases where $\chi \deg(v) + \rho \leq 0$ or $\chi \deg^+(v) + \rho \leq 0$ (from the descriptions below we see that the hooking networks studied are connected and so the vertex $v$ has degree $\deg(v) > 0$; we also see below that all vertices $v$ that are candidates for being a latch in the bipolar networks studied satisfy $\deg^+(v) > 0$). For a positive integer $k$, we let $w_k := \chi k + \rho$.

\subsubsection{Hooking networks}\label{sec:hooking}

Let $\mathcal{C} = \{ G_1, G_2, \ldots, G_m\}$ be a set of connected graphs, each with at least 2 vertices, and each with a labelled vertex $h_i$. We allow for the graphs to contain self-loops and multiple edges. The graph $G_i$ is called a {\em block}, and the vertex $h_i$ is called its {\em hook}. Each block $G_i$ is also assigned a positive probability $p_i$ such that $p_1 + p_2 + \cdots + p_m = 1$. For example, consider the set of blocks in Figure \ref{fig:hookseeds}, with their hooks labelled and their probabilities written underneath. 

\begin{figure}[h!]
\centering
\begin{tikzpicture}[scale=0.75]
\filldraw(2,1) circle (3pt) node(t1){} node[above right]{$h_1$};
\filldraw(1,0) circle (3pt) node(l1){};
\filldraw(3,0) circle (3pt) node(r1){};
\node at (2,-1) {$G_1$};
\node at (2,-1.6) {$p_1 = 1/6$};

\draw (l1.center) -- (t1.center)-- (r1.center);

\filldraw(6,1) circle (3pt) node(t2){}node[above right]{$h_2$};
\filldraw(5,0) circle (3pt) node(l2){};
\filldraw(5.6666,0) circle (3pt) node(bl2){};
\filldraw(6.3333, 0) circle (3pt) node(br2){};
\filldraw(7,0) circle (3pt) node(r2){};
\node at (6,-1) {$G_2$};
\node at (6,-1.6) {$p_1 = 1/3$};

\draw (l2.center) -- (t2.center) -- (r2.center);
\draw (bl2.center) -- (t2.center) -- (br2.center);

\filldraw(10,1) circle (3pt) node(t3){}node[above right]{$h_3$};
\filldraw(9,0) circle (3pt) node(l3){};
\filldraw(9.6666,0) circle (3pt) node(bl3){};
\filldraw(10.3333, 0) circle (3pt) node(br3){};
\filldraw(11,0) circle (3pt) node(r3){};
\node at (10,-1) {$G_3$};
\node at (10,-1.6) {$p_1 = 1/6$};

\draw (l3.center) -- (bl3.center) -- (br3.center) -- (r3.center);
\draw (bl3.center) -- (t3.center) -- (br3.center);

\filldraw(14,1) circle (3pt) node(t4){}node[above right]{$h_4$};
\filldraw(13,0) circle (3pt) node(l4){};
\filldraw(13.6666,0) circle (3pt) node(bl4){};
\filldraw(14.3333, 0) circle (3pt) node(br4){};
\filldraw(15,0) circle (3pt) node(r4){};
\node at (14,-1) {$G_4$};
\node at (14,-1.6) {$p_1 = 1/3$};

\draw (br4.center) -- (bl4.center) -- (t4.center) -- (r4.center) -- (br4.center) -- (t4.center) -- (l4.center) -- (bl4.center);
\draw (l4.center) arc (225: 315:1.4142);

\end{tikzpicture}
\caption{A set of simple graphs as blocks}
\label{fig:hookseeds}
\end{figure}
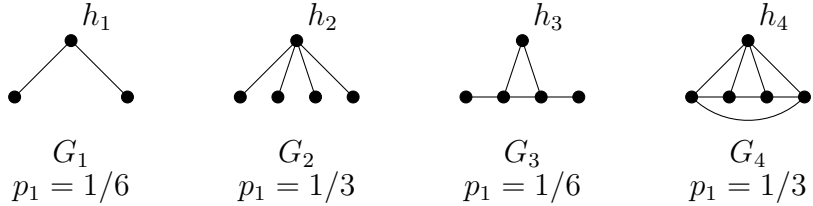

Let $\chi$ and $\rho$ be real numbers satisfying the conditions set above. A sequence of hooking networks $\mathcal{G}_0, \mathcal{G}_1, \mathcal{G}_2, \ldots$ is constructed as follows: one of the blocks $G_i$ is chosen, and we set $\mathcal{G}_0$ to be a copy of $G_i$ (the choice of the first block does not need to be done at random for our methods to work). The vertex $H$ that corresponds to the hook of this first block copied to make $\mathcal{G}_0$ is called the {\em master hook} of the hooking networks constructed afterwards; when all the blocks are trees the master hook acts as the root of the network. Recursively for $n \geq 1$, the hooking network $\mathcal{G}_n$ is constructed from $\mathcal{G}_{n-1}$ by first choosing a latch $v$ at random proportionally to $\chi \deg(v) + \rho$ amongst all the vertices of $\mathcal{G}_{n-1}$, that is, with probability 
\begin{equation}\label{eq:hooklatchprob}
\frac{\chi \deg(v) + \rho}{\sum_{u \in V(\mathcal{G}_{n-1})} \chi \deg(u) + \rho},
\end{equation}
where $V(\mathcal{G}_{n-1})$ is the vertex set of $\mathcal{G}_{n-1}$. Once the latch is chosen, a block $G_i$ is chosen according to its probability $p_i$. A copy of $G_i$ is attached to $\mathcal{G}_{n-1}$ by fusing together the latch $v$ with the hook $h_i$ of the copy of $G_i$; that is, $h_i$ is deleted and edges are drawn from $v$ to the former neighbours of $h_i$. Figure \ref{fig:hookingnetworks} is a sequence of hooking networks constructed from the set of blocks in Figure \ref{fig:hookseeds} by taking a copy of $G_3$ and attaching copies of $G_4$, then $G_2$, and finally a copy of $G_1$. The master hook of the network is labelled $H$, and at each step the vertex chosen to be the latch is denoted by $\ast$. 

\begin{figure}[h!]
\centering
\begin{tikzpicture}[scale=0.85]

\filldraw(3,8) circle (3pt) node(t1){} node[above right]{$H$};
\filldraw(2,7) circle (3pt) node(l1){};
\filldraw(2.6666,7) circle (3pt) node(bl1){};
\filldraw(3.3333, 7) circle (3pt) node(br1){};
\node(r1) at (4,7){\LARGE{$\ast$}};
\draw (l1.center) -- (bl1.center) -- (br1.center) -- (r1.center);
\draw (bl1.center) -- (t1.center) -- (br1.center);

\node at (3,5.1) {$\mathcal{G}_0$};

\filldraw(9,8) circle (3pt) node(t1){}node[above right]{$H$};
\node(l1) at  (8,7){\LARGE{$\ast$}};
\filldraw(8.6666,7) circle (3pt) node(bl1){};
\filldraw(9.3333, 7) circle (3pt) node(br1){};
\filldraw(10,7) circle (3pt) node(r1){};
\draw (l1.center) -- (bl1.center) -- (br1.center) -- (r1.center);
\draw (bl1.center) -- (t1.center) -- (br1.center);

\filldraw(9,6) circle (3pt) node(l2){};
\filldraw(9.6666, 6) circle (3pt) node (bl2){};
\filldraw(10.3333, 6) circle (3pt) node (br2){};
\filldraw(11,6) circle (3pt) node (r2){};
\draw (br2.center) -- (bl2.center) -- (r1.center) -- (br2.center) -- (r2.center) -- (r1.center) -- (l2.center) -- (bl2.center);
\draw (l2.center) arc (225: 315:1.4142);

\node at (9,5.1) {$\mathcal{G}_1$};

\filldraw(3,4) circle (3pt) node(t1){}node[above right]{$H$};
\filldraw(2,3) circle (3pt) node(l1){};
\filldraw(2.6666,3) circle (3pt) node(bl1){};
\filldraw(3.3333, 3) circle (3pt) node(br1){};
\node(r1) at (4,3){\LARGE{$\ast$}};
\draw (l1.center) -- (bl1.center) -- (br1.center) -- (r1.center);
\draw (bl1.center) -- (t1.center) -- (br1.center);

\filldraw(3,2) circle (3pt) node(l2){};
\filldraw(3.6666, 2) circle (3pt) node (bl2){};
\filldraw(4.3333, 2) circle (3pt) node (br2){};
\filldraw(5,2) circle (3pt) node (r2){};
\draw (br2.center) -- (bl2.center) -- (r1.center) -- (br2.center) -- (r2.center) -- (r1.center) -- (l2.center) -- (bl2.center);
\draw (l2.center) arc (225: 315:1.4142);

\filldraw(1,4) circle (3pt) node(l3){};
\filldraw(1, 3.3333) circle (3pt) node(bl3){};
\filldraw(1, 2.6666) circle (3pt) node(br3){};
\filldraw(1,2) circle (3pt) node(r3){};
\draw (l3.center) -- (l1.center) -- (r3.center);
\draw (bl3.center) -- (l1.center) -- (br3.center);

\node at (3,1.1) {$\mathcal{G}_2$};

\filldraw(9,4) circle (3pt) node(t1){}node[above right]{$H$};
\filldraw(8,3) circle (3pt) node(l1){};
\filldraw(8.6666,3) circle (3pt) node(bl1){};
\filldraw(9.3333, 3) circle (3pt) node(br1){};
\filldraw(10,3) circle (3pt) node(r1){};
\draw (l1.center) -- (bl1.center) -- (br1.center) -- (r1.center);
\draw (bl1.center) -- (t1.center) -- (br1.center);

\filldraw(9,2) circle (3pt) node(l2){};
\filldraw(9.6666, 2) circle (3pt) node (bl2){};
\filldraw(10.3333, 2) circle (3pt) node (br2){};
\filldraw(11,2) circle (3pt) node (r2){};
\draw (br2.center) -- (bl2.center) -- (r1.center) -- (br2.center) -- (r2.center) -- (r1.center) -- (l2.center) -- (bl2.center);
\draw (l2.center) arc (225: 315:1.4142);

\filldraw(7,4) circle (3pt) node(l3){};
\filldraw(7, 3.3333) circle (3pt) node(bl3){};
\filldraw(7, 2.6666) circle (3pt) node(br3){};
\filldraw(7,2) circle (3pt) node(r3){};
\draw (l3.center) -- (l1.center) -- (r3.center);
\draw (bl3.center) -- (l1.center) -- (br3.center);

\filldraw(11.414, 3) circle (3pt) node (r4){};
\filldraw(10, 4.414) circle (3pt) node (l4){};
\draw (r4.center) -- (r1.center) -- (l4.center);

\node at (9,1.1) {$\mathcal{G}_3$};
\end{tikzpicture}
\caption{A sequence of hooking networks grown from the blocks $G_1, G_2, G_3$ and $G_4$ of Figure \ref{fig:hookseeds}}
\label{fig:hookingnetworks}
\end{figure}
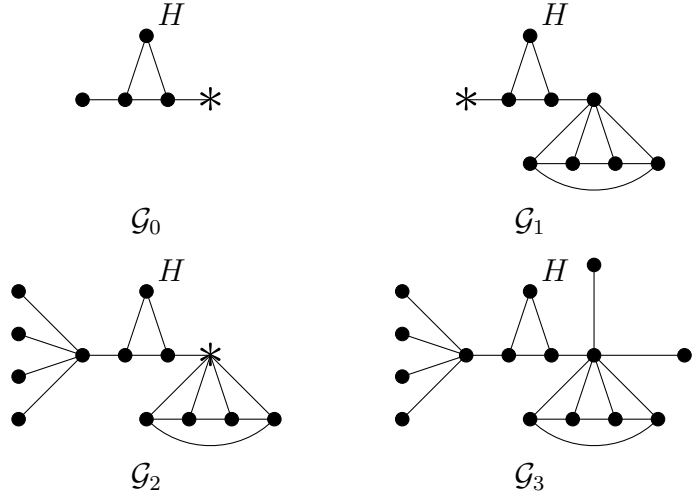

\subsubsection{Bipolar networks}\label{sec:bipolar}

For a vertex $v$ in a directed graph $B$, we denote by $\deg^-(v)$ the indegree of $v$: the number of arcs leading into $v$, and by $\deg^+(v)$ the outdegree of $v$: the number of arcs leading out of $v$. If $\deg^-(v) = 0$ then $v$ is called a {\em source}, and if $\deg^+(v) = 0$, $v$ is called a {\em sink}. Chen and Mahmoud \cite{CHMA:16} define a {\em bipolar directed graph} $B$ to be a directed acyclic graph containing a unique source $N$ called the {\em north pole} of $B$, a unique sink $S$ called the {\em south pole} of $B$, and a directed path from every vertex $v \neq S$ in $B$ to $S$. The methods presented here also apply to a more relaxed definition of bipolar directed graphs: connected directed graphs with a single source and a single sink. Let $\mathcal{C} = \{B_1, B_2, \ldots, B_m\}$ be a set of bipolar directed graphs, each with their north pole $N_i$ and south pole $S_i$ identified. Each $B_i$ is called a {\em block}, and is assigned a probability $p_i$ such that $p_1 + p_2 + \cdots + p_m = 1$. For example, consider the set of blocks in Figure \ref{fig:bipolarseeds}, with their north and south poles labelled as well as their probabilities. 

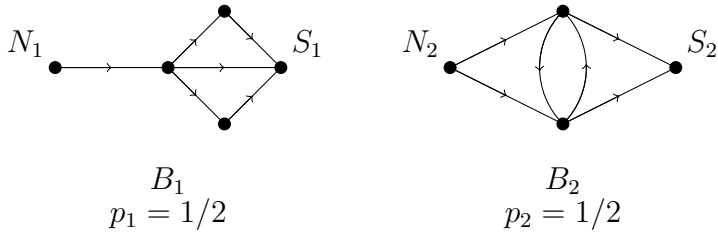
\begin{figure}[h!]
\centering
\begin{tikzpicture}[scale=0.75]

\filldraw(1,2) circle (3pt) node(n1){} node[above left]{$N_1$};
\node(nb1) at (3.5, 1.5){};
\filldraw(4, 1) circle (3pt) node (b1){};
\node(nt1) at (3.5, 2.5){};
\filldraw(4, 3) circle (3pt) node(t1){};
\node(bm1) at (4.5, 1.5){};
\node(nm1) at (2,2){};
\filldraw(3,2) circle (3pt) node (m1){};
\node(tm1) at (4.5, 2.5){};
\filldraw(5,2) circle (3pt) node (s1){} node[above right]{$S_1$};
\node(ms1) at (4,2){};
\draw[->] (m1.center) -- (nb1.center);
\draw[->] (m1.center) -- (nt1.center);
\draw[->] (n1.center) -- (nm1.center);
\draw[->] (b1.center) -- (bm1.center);
\draw[->] (t1.center) -- (tm1.center);
\draw[->] (m1.center) -- (ms1.center);
\draw (n1.center) -- (m1.center) -- (s1.center);
\draw (m1.center) -- (t1.center) -- (s1.center);
\draw (m1.center) -- (b1.center) -- (s1.center);
\node at (3,0) {$B_1$};
\node at (3, -0.6) {$p_1 = 1/2$};

\filldraw(8,2) circle (3pt) node (n2){} node[above left]{$N_2$};
\node(nt2) at (9,2.5){};
\filldraw(10,3) circle (3pt) node (t2){};
\node(nb2) at (9,1.5){};
\filldraw(10,1) circle (3pt) node (b2){};
\node(ts2) at (11,2.5){};
\node(bs2) at (11,1.5){};
\filldraw(12,2) circle (3pt) node (s2){} node[above right]{$S_2$};
\draw[->] (n2.center) -- (nt2.center);
\draw[->] (n2.center) -- (nb2.center);
\draw[->] (t2.center) -- (ts2.center);
\draw[->] (b2.center) -- (bs2.center);
\draw[->] (t2.center) arc (135:182.5:1.4142);
\draw[->] (b2.center) arc (-45:2.5:1.4142);
\draw (n2.center) -- (t2.center) -- (s2.center);
\draw (n2.center) -- (b2.center) -- (s2.center);
\draw (t2.center) arc (135: 225: 1.4142);
\draw (b2.center) arc (-45: 45:1.4142);
\node at (10,0) {$B_2$};
\node at (10, -0.6) {$p_2  = 1/2$};

\end{tikzpicture}
\caption{A set of bipolar directed graphs as blocks}
\label{fig:bipolarseeds}
\end{figure}

Once again, we let $\chi$ and $\rho$ be real numbers satisfying the conditions set at the beginning of this section. We choose a block $B_i$ and set the bipolar network $\mathcal{B}_0$ to be a copy of $B_i$ (once again, the choice of the first block need not be made at random). The vertices corresponding to the north and south poles of $\mathcal{B}_0$ serve as the master source $N$ and master sink $S$ respectively of the bipolar networks constructed afterwards. For $n \geq 1$, the bipolar network $\mathcal{B}_n$ is constructed from $\mathcal{B}_{n-1}$ in a manner similar to that of hooking networks. First, a latch $v$ is chosen proportionally to $\chi \deg^+(v) + \rho$ amongst all the vertices in $\mathcal{B}_{n-1}$ that are not the master sink, that is, with probability 
\begin{equation}\label{eq:biplatchprob}
\frac{\chi \deg^+(v) + \rho}{\sum_{u \in V(\mathcal{B}_{n-1})\setminus\{S\}}\chi \deg^+(u) + \rho},
\end{equation}
where $V(\mathcal{B}_{n-1})$ is the vertex set of $\mathcal{B}_{n-1}$. Once the is latch chosen, one of the arcs $(v,u)$ leading out of $v$ is chosen uniformly at random amongst all the arcs leading out of $v$, and finally a block $B_i$ is chosen according to its probability $p_i$. The arc $(v,u)$ is deleted, and a copy of the block $B_i$ is added by fusing the north pole $N_i$ with $v$, and fusing the south pole $S_i$ with $u$. We never allow the master sink to be chosen as a latch (since it has no arcs leading out of it). Figure \ref{fig:bipolarnetworks} is a sequence of bipolar networks constructed from the blocks in Figure \ref{fig:bipolarseeds}. The master source $N$ and the master sink $S$ are labelled, and at each step, the latch $v$ is denoted by $\ast$, and the arc $(v,u)$ to be removed is dashed. 

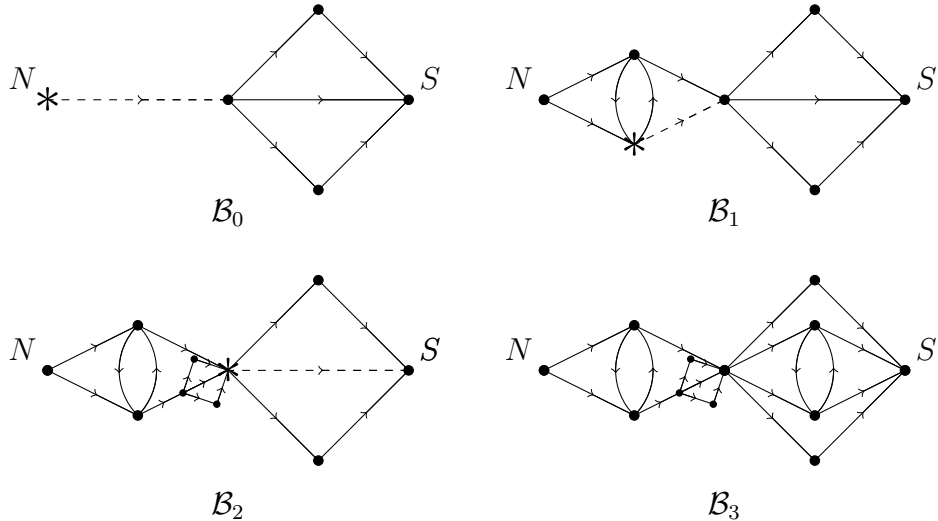
\begin{figure}[h!]
\centering
\begin{tikzpicture}[scale = 0.6]

\filldraw(9,4) circle (3pt) node(n1){}node[above right]{$S$};
\node(nb1) at (8, 3){};
\filldraw(7, 2) circle (3pt) node (b1){};
\node(nt1) at (8, 5){};
\filldraw(7, 6) circle (3pt) node(t1){};
\node(bm1) at (6, 3){};
\node(nm1) at (7,4){};
\filldraw(5,4) circle (3pt) node (m1){};
\node(tm1) at (6, 5){};
\node(s1) at (1,4) {\LARGE $\ast$};
\node at (1,4)[above left]{$N$};
\node(ms1) at (3,4){};
\draw[-<] (n1.center) -- (nb1.center);
\draw[-<] (n1.center) -- (nt1.center);
\draw[-<] (n1.center) -- (nm1.center);
\draw[-<] (b1.center) -- (bm1.center);
\draw[-<] (t1.center) -- (tm1.center);
\draw[dashed,-<] (m1.center) -- (ms1.center);
\draw (n1.center) -- (m1.center);
\draw[dashed] (m1.center) -- (s1.center);
\draw (n1.center) -- (t1.center) -- (m1.center);
\draw (n1.center) -- (b1.center) -- (m1.center);
\node at (5,1.5) {$\mathcal{B}_0$};

\filldraw(20,4) circle (3pt) node(n1){}node[above right]{$S$};
\node(nb1) at (19, 3){};
\filldraw(18, 2) circle (3pt) node (b1){};
\node(nt1) at (19, 5){};
\filldraw(18, 6) circle (3pt) node(t1){};
\node(bm1) at (17, 3){};
\node(nm1) at (18,4){};
\filldraw(16,4) circle (3pt) node (m1){};
\node(tm1) at (17, 5){};
\filldraw(12,4) circle (3pt) node (s1){}node[above left]{$N$};
\node(ms1) at (14,4){};
\draw[-<] (n1.center) -- (nb1.center);
\draw[-<] (n1.center) -- (nt1.center);
\draw[-<] (n1.center) -- (nm1.center);
\draw[-<] (b1.center) -- (bm1.center);
\draw[-<] (t1.center) -- (tm1.center);
\draw (n1.center) -- (m1.center); 
\draw (n1.center) -- (t1.center) -- (m1.center);
\draw (n1.center) -- (b1.center) -- (m1.center);

\node(n2) at (16,4){};
\node(nt2) at (15,4.5){};
\filldraw(14,5) circle (3pt) node (t2){};
\node(nb2) at (15,3.5){};
\node(b2) at (14,3) {\LARGE $\ast$};
\node(ts2) at (13,4.5){};
\node(bs2) at (13,3.5){};
\filldraw(12,4) circle (3pt) node (s2){};
\draw[-<] (n2.center) -- (nt2.center);
\draw[dashed, -<] (n2.center) -- (nb2.center);
\draw[-<] (t2.center) -- (ts2.center);
\draw[-<] (b2.center) -- (bs2.center);
\draw[->] (t2.center) arc (135:182.5:1.4142);
\draw[->] (b2.center) arc (-45:2.5:1.4142);
\draw  (b2.center) -- (s2.center);
\draw[dashed] (n2.center) -- (b2.center);
\draw (n2.center) -- (t2.center) -- (s2.center);
\draw (t2.center) arc (135: 225: 1.4142);
\draw (b2.center) arc (-45: 45:1.4142);
\node at (16,1.5) {$\mathcal{B}_1$};

\filldraw(9,-2) circle (3pt) node(n1){}node[above right]{$S$};
\node at (9,-2)[above right]{$S$};
\node(nb1) at (8, -3){};
\filldraw(7, -4) circle (3pt) node (b1){};
\node(nt1) at (8, -1){};
\filldraw(7, 0) circle (3pt) node(t1){};
\node(bm1) at (6, -3){};
\node(nm1) at (7,-2){};
\node(m1) at (5,-2) {\LARGE $\ast$};
\node(tm1) at (6, -1){};
\filldraw(1,-2) circle (3pt) node (s1){}node[above left]{$N$};
\node(ms1) at (3,-2){};
\draw[-<] (n1.center) -- (nb1.center);
\draw[-<] (n1.center) -- (nt1.center);
\draw[dashed, -<] (n1.center) -- (nm1.center);
\draw[-<] (b1.center) -- (bm1.center);
\draw[-<] (t1.center) -- (tm1.center);
\draw[dashed] (n1.center) -- (m1.center); 
\draw (n1.center) -- (t1.center) -- (m1.center);
\draw (n1.center) -- (b1.center) -- (m1.center);

\node (n2) at (5,-2){};
\node(nt2) at (4,-1.5){};
\filldraw(3,-1) circle (3pt) node (t2){};
\node(nb2) at (4,-2.5){};
\filldraw(3,-3) circle (3pt) node (b2){};
\node(ts2) at (2,-1.5){};
\node(bs2) at (2,-2.5){};
\filldraw(1,-2) circle (3pt) node (s2){};
\draw[-<] (n2.center) -- (nt2.center);
\draw[-<] (t2.center) -- (ts2.center);
\draw[-<] (b2.center) -- (bs2.center);
\draw[->] (t2.center) arc (135:182.5:1.4142);
\draw[->] (b2.center) arc (-45:2.5:1.4142);
\draw (n2.center) -- (t2.center) -- (s2.center);
\draw (n2.center) -- (b2.center) -- (s2.center);
\draw (t2.center) arc (135: 225: 1.4142);
\draw (b2.center) arc (-45: 45:1.4142);

\filldraw(4.75, -2.75) circle (2pt) node (b3){};
\filldraw(4.25, -1.75) circle (2pt) node (t3){};
\filldraw(4,-2.5) circle (2pt) node (m3){};
\node(nb3) at (4.375, -2.625){};
\node(nt3) at (4.125, -2.125){};
\node(nm3) at (3.5, -2.75){};
\node(ms3) at (4.5, -2.25){};
\node(tm3) at (4.625, -1.875){};
\node(bm3) at (4.875, -2.375){};
\draw[->] (m3.center) -- (nb3.center);
\draw[->] (m3.center) -- (nt3.center);
\draw[-<] (n2.center) -- (nm3.center);
\draw[->] (nb2.center) -- (ms3.center);
\draw[->] (t3.center) -- (tm3.center);
\draw[->] (b3.center) -- (bm3.center);
\draw (m3.center) -- (t3.center) -- (m1.center);
\draw (m3.center) -- (b3.center) -- (m1.center);
\node at (5,-5) {$\mathcal{B}_2$};

\filldraw(20,-2) circle (3pt) node(n1){}node[above right]{$S$};
\node(nb1) at (19, -3){};
\filldraw(18, -4) circle (3pt) node (b1){};
\node(nt1) at (19, -1){};
\filldraw(18, 0) circle (3pt) node(t1){};
\node(bm1) at (17, -3){};
\node(nm1) at (18,-2){};
\filldraw(16,-2) circle (3pt) node (m1){};
\node(tm1) at (17, -1){};
\filldraw(12,-2) circle (3pt) node (s1){}node[above left]{$N$};
\node(ms1) at (14,-2){};
\draw[-<] (n1.center) -- (nb1.center);
\draw[-<] (n1.center) -- (nt1.center);
\draw[-<] (b1.center) -- (bm1.center);
\draw[-<] (t1.center) -- (tm1.center);
\draw (n1.center) -- (t1.center) -- (m1.center);
\draw (n1.center) -- (b1.center) -- (m1.center);

\filldraw(16,-2) circle (3pt) node (n2){};
\node(nt2) at (15,-1.5){};
\filldraw(14,-1) circle (3pt) node (t2){};
\node(nb2) at (15,-2.5){};
\filldraw(14,-3) circle (3pt) node (b2){};
\node(ts2) at (13,-1.5){};
\node(bs2) at (13,-2.5){};
\filldraw(12,-2) circle (3pt) node (s2){};
\draw[-<] (n2.center) -- (nt2.center);
\draw[-<] (t2.center) -- (ts2.center);
\draw[-<] (b2.center) -- (bs2.center);
\draw[->] (t2.center) arc (135:182.5:1.4142);
\draw[->] (b2.center) arc (-45:2.5:1.4142);
\draw (n2.center) -- (t2.center) -- (s2.center);
\draw (n2.center) -- (b2.center) -- (s2.center);
\draw (t2.center) arc (135: 225: 1.4142);
\draw (b2.center) arc (-45: 45:1.4142);

\filldraw(15.75, -2.75) circle (2pt) node (b3){};
\filldraw(15.25, -1.75) circle (2pt) node (t3){};
\filldraw(15,-2.5) circle (2pt){};
\node(nb3) at (15.375, -2.625){};
\node(nt3) at (15.125, -2.125){};
\node(nm3) at (14.5, -2.75){};
\node(ms3) at (15.5, -2.25){};
\node(tm3) at (15.625, -1.875){};
\node(bm3) at (15.875, -2.375){};
\draw[->] (nb2.center) -- (nb3.center);
\draw[->] (nb2.center) -- (nt3.center);
\draw[-<] (n2.center) -- (nm3.center);
\draw[->] (nb2.center) -- (ms3.center);
\draw[->] (t3.center) -- (tm3.center);
\draw[->] (b3.center) -- (bm3.center);
\draw (nb2.center) -- (t3.center) -- (m1.center);
\draw (nb2.center) -- (b3.center) -- (m1.center);

\filldraw(20,-2) circle (3pt) node (n4){};
\node(nt4) at (19,-1.5){};
\filldraw(18,-1) circle (3pt) node (t4){};
\node(nb4) at (19,-2.5){};
\filldraw(18,-3) circle (3pt) node (b4){};
\node(ts4) at (17,-1.5){};
\node(bs4) at (17,-2.5){};
\filldraw(16,-2) circle (3pt) node (s4){};
\draw[-<] (n4.center) -- (nt4.center);
\draw[-<] (n4.center) -- (nb4.center);
\draw[-<] (t4.center) -- (ts4.center);
\draw[-<] (b4.center) -- (bs4.center);
\draw[->] (t4.center) arc (135:182.5:1.4142);
\draw[->] (b4.center) arc (-45:2.5:1.4142);
\draw (n4.center) -- (t4.center) -- (s4.center);
\draw (n4.center) -- (b4.center) -- (s4.center);
\draw (t4.center) arc (135: 225: 1.4142);
\draw (b4.center) arc (-45: 45:1.4142);
\node at (16,-5) {$\mathcal{B}_3$};

\end{tikzpicture}
\caption{A sequence of bipolar networks grown from the blocks $B_1$ and $B_2$ from Figure \ref{fig:bipolarseeds}}
\label{fig:bipolarnetworks}
\end{figure}

Previously, Chen and Mahmoud \cite{CHMA:16} studied what they called {\em self-similar bipolar networks}. These are bipolar networks grown from a single bipolar directed graph as the only block. At each step in the growth of their networks, an arc $(v,u)$ is chosen uniformly at random amongst all the arcs to be deleted before being replaced with a copy of the block. This is equivalent to choosing $v$ proportionally to its outdegree $\deg^+(v)$, and then choosing an arc $(v,u)$ uniformly at random amongst all the arcs leading out of $v$. Therefore, the model of bipolar networks introduced here extends their model.

\subsection{Main results}\label{sec:main}

Before we state the main results, we need a useful definition. In the interest of length, the notation {\em (out)degree} is used in the following discussion, and is interpreted as {\em degree} for hooking networks and {\em outdegree} for bipolar networks. 

Depending on the set of blocks that are used to grow the hooking networks or bipolar networks, it is possible for some positive integers to never appear as the (out)degree of a vertex in the network, while some integers are only the (out)degree of at most one vertex at some point in the growth of the network. By ignoring these so-called {\em nonessential (out)degrees}, formally defined below, the proofs using P\'olya urns are simplified. We also show by a simple argument below (see Proposition \ref{pro:adm}) that only the master hook or master source may have a nonessential (out)degree. Excluding this single vertex from the (out)degree distributions does not affect the asymptotic behaviour of these distributions.


\begin{definition}\label{def:admissible}
Given a set $\mathcal{C}$ of blocks, a (strictly) positive integer $k$ is called an {\em essential (out)degree} if with positive probability, there is some $n$ so that the $n$-th iteration of the network grown out of $\mathcal{C}$ has at least two vertices with (out)degree $k$. A positive integer is called a {\em nonessential (out)degree} if it is not an essential (out)degree. 
\end{definition}

\begin{remark}
Our definition of essential (out)degrees differs slightly from the definition of {\em admissible} (out)degrees used in \cite{CHMA:16} and \cite{MAHM:19}, where any (out)degree that may appear in the network is considered an admissible (out)degree.
\end{remark}

In the example of hooking networks grown in Section \ref{sec:hooking} from the blocks in Figure \ref{fig:hookseeds}, all of the hooks of the blocks have even degrees, and all other vertices in the blocks have odd degrees. As a result, during the growth of the hooking networks, only the master hook has even degree, while every other vertex has odd degree (as is evidenced by the hooking networks in Figure \ref{fig:hookingnetworks}). In that case, the odd numbers are essential degrees, and the even numbers are nonessential. 

\begin{proposition}\label{pro:adm}
The only vertex in a hooking network (or bipolar network) that can have a nonessential (out)degree is the master hook (or master source) of the network. 
\end{proposition}

\begin{proof}
We only prove the proposition for hooking networks; the argument is similar for bipolar networks. 

Suppose there is a positive probability that a vertex $v$ which is not the master hook has degree $k$ in the hooking network $\mathcal{G}_{n}$, and without loss of generality let $n$ be the smallest number for which $\mathcal{G}_n$ has a vertex $v$ with degree $k$. We will show that with positive probability, another vertex that is not the master hook will have degree $k$ in a later iteration of the hooking network. 

The vertex $v$ first appears in the network as a non-hook vertex with degree $k_0$ of a newly added block; say the block was $G_{i_0}$ and $v$ is a copy of the vertex $v_0$ in $G_{i_0}$. If $k_0 \neq k$, then that means hooks of other blocks were fused to $v$, say the first hook fused to $v$ belonged to $G_{i_1}$, the second belonged to $G_{i_2}$, and so on until the last hook fused to $v$ which belonged to $G_{i_r}$ (which was the last block added to create $\mathcal{G}_n$). With positive probability, a copy of the block $G_{i_0}$ is joined to $\mathcal{G}_{n}$ by  fusing the hook of $G_{i_0}$ with a vertex that is not $v$, say the master hook. Let $u$ be the newly added vertex in the hooking network that is a copy of $v_0$ in $G_{i_0}$. For $j=1, \ldots, r$, there is a positive probability that the block $G_{i_j}$ is added to the hooking network $\mathcal{G}_{n + j}$ by fusing the hook of $G_{i_j}$ with $u$. In this case, $u$ has degree $k$ in $\mathcal{G}_{n+r+1}$, and so there is a positive probability that 2 vertices ($v$ and $u$) have degree $k$ in $\mathcal{G}_{n + r + 1}$. Therefore, $k$ is an essential degree. 
\end{proof}

Also note that in the case of bipolar networks, only the master sink of the network has outdegree 0, and we therefore ignore this vertex completely.

\subsubsection{Main results for hooking networks}\label{sec:hookingthm}

Let $\mathcal{C} = \{G_1, G_2, \ldots, G_m\}$ be a set of blocks, each with an identified hook $h_i$, and let $\mathcal{G}_0, \mathcal{G}_1, \mathcal{G}_2, \ldots$ be a sequence of hooking networks grown from $\mathcal{C}$, with the master hook of the network labelled $H$. We allow for the latches and the blocks added at each step to be chosen in the manner laid out in Section \ref{sec:networks} (that is, with linear preferential attachment with parameters $\chi$ and $\rho$, and probabilities $p_i$ assigned to each block $G_i$). For a positive integer $r$, let 
\[ k_1 < k_2 < \cdots < k_r \]
be the first $r$ essential degrees. For a positive integer $k$, recall that $w_k = \chi k + \rho$. For each block $G_i$ in the set $\mathcal{C}$, let $V(G_i)$ be its vertex set. For a positive integer $k$, define 
\begin{equation}\label{eq:fhook}
 f(k) := \sum_{G_i \in \mathcal{C}} p_i \cdot | \{ v \in V(G_i) \setminus \{h_i\} : \deg(v) = k\}| 
 \end{equation}
and
\begin{equation}\label{eq:ghook}
g(k) := \sum_{\substack{G_i \in \mathcal{C} \\ \deg(h_i) = k} }\!\!\!\!p_i.
\end{equation}
The value $f(k)$ is the expected number of new vertices of degree $k$ (that are not hooks) added at any step, and $g(k)$ is the probability that the degree of the latch chosen at any step is increased by $k$ after fusing with the hook of the newly attached block. For example, for the blocks in Figure \ref{fig:hookseeds} we have that $f(1) = 2$ and $f(3) = 5/3$, while $g(2) = 1/3$ and $g(4) = 2/3$. 
Define 
\begin{equation}\label{eq:lambdahook}
 \lambda_1 := \sum_{k \geq 1} (w_k f(k) + \chi k g(k)).
 \end{equation}
The value $\lambda_1$ is the expected change in the denominator of ($\ref{eq:hooklatchprob}$) at each step in the growth of the hooking network. For our running example of hooking networks grown from the blocks in Figure \ref{fig:hookseeds}, if we let $\chi =1$ and $\rho =0$, then 
\begin{equation}\label{eq:lambdahookexample}
\lambda_1 = \frac{31}{3}.
\end{equation}
Let $\nu_1 := f(k_1)/(\lambda_1 + w_{k_1})$, and define recursively for $i=2, \ldots, r$
\begin{equation}\label{eq:vhookdef}
 \nu_i := \frac{1}{\lambda_1 + w_{k_i}}\left(f(k_i) + \sum_{j=1}^{i-1}w_{k_j}g(k_i - k_j)\nu_j\right).
 \end{equation}
The value $\lambda_1 \nu_i$ is the limit of the expected proportion of vertices with degree $k_i$ (see Remark \ref{rem:hookprop} below). Let $\nu$ be the vector 
\begin{equation}\label{eq:vhook}
\nu := (\nu_1, \nu_2, \ldots, \nu_r). 
\end{equation}
For our running example of hooking networks grown from the blocks in Figure \ref{fig:hookseeds} with $\chi =1$ and $\rho =0$, and if we let $r=3$, then the first 3 essential degrees are $1,3,5$ (recall that only odd numbers are essential in this example), and 
\begin{equation}\label{eq:vhookexample}
\nu = \left(\frac{6}{34}, \frac{11}{85}, \frac{63}{3910}\right).
\end{equation}

We have the following multivariate normal limit law for the degrees of hooking networks.

\begin{theorem}\label{thm:hooking}
Let $\mathcal{X}_n = (X_{n,1}, X_{n,2}, \ldots, X_{n,r}),$ where $X_{n,i}$ is the number of vertices with essential degree $k_i$ in $\mathcal{G}_n$, where $\mathcal{G}_n$ is a hooking network grown from the set of blocks $\mathcal{C}$ using linear preferential attachment with parameters $\chi$ and $\rho$. Let $\lambda_1$ be defined as in (\ref{eq:lambdahook}) and let $\nu$ be the vector defined in (\ref{eq:vhookdef}) and (\ref{eq:vhook}). Then 
\begin{equation}\label{eq:hookconv}
n^{-1/2}(\mathcal{X}_n - n\lambda_1\nu) \xrightarrow{d} \mathcal{N}(0,\Sigma)
\end{equation} 
for some covariance matrix $\Sigma$.
\end{theorem}

\begin{remark}\label{rem:hookprop}
From (\ref{eq:hookconv}), we see an immediate weak law of large numbers, 
\begin{equation}\label{eq:LLNhook}
\mathcal{X}_n/n \xrightarrow{p} \lambda_1\nu. 
\end{equation}
Furthermore, since the number of blocks is finite and each block has a finite number of vertices, there is a constant $C$ such that $0 \leq X_{n,i} \leq Cn$ for all $i=1,2, \ldots, r$ and all $n$. Therefore, the random vectors $\mathcal{X}_n/n$ are uniformly integrable which, along with (\ref{eq:LLNhook}), imply
\[ \mathbb{E}\mathcal{X}_n/n \rightarrow \lambda_1 \nu.\]
The convergence in (\ref{eq:LLNhook}) also holds almost surely (see Remark \ref{rem:asurn}).
\end{remark}


In some special cases, we can say even more about the convergence in (\ref{eq:hookconv}). For each block $G_i$, let $E(G_i)$ be the set of edges of $G_i$, and let
\begin{equation}\label{eq:sforhooks}
 s_i := \sum_{u \in V(G_i)}(\chi\deg(u) + \rho) - \rho = 2\chi |E(G_i)|+ \rho(|V(G_i)| - 1). 
 \end{equation}

\begin{corollary}\label{thm:hookingspecial}
Let $\mathcal{X}_n = (X_{n,1}, X_{n,2}, \ldots, X_{n,r}),$ where $X_{n,i}$ is the number of vertices with essential degree $k_i$ in $\mathcal{G}_n$, where $\mathcal{G}_n$ is a hooking network grown from the set of blocks $\mathcal{C}$ using linear preferential attachment with parameters $\chi$ and $\rho$. Let $\lambda_1$ be defined as in (\ref{eq:lambdahook}), let $\nu$ be the vector defined in (\ref{eq:vhookdef}) and (\ref{eq:vhook}), and let $s_i$ be defined as in (\ref{eq:sforhooks}) for each block $G_i$. Suppose that there exists a constant $s$ so that $s_i = s$ for all blocks $G_i$. Then the convergence (\ref{eq:hookconv}) holds in all moments. In particular, $n^{-1/2}(\mathbb{E}X_n - n\lambda_1\nu) \rightarrow 0$, and so $n\lambda_1\nu$ in (\ref{eq:hookconv}) can be replaced by $\mathbb{E}X_n$.
\end{corollary}

There are several cases where Corollary \ref{thm:hookingspecial} applies. An obvious example is when there is only one block to choose from. Other examples include when $\chi = 0$ and all the graphs have the same number of vertices, or when $\rho = 0$ and all the graphs have the same number of edges. 

To compare Theorem \ref{thm:hooking} with previous results on random recursive trees and preferential attachment trees, consider a hooking network grown from $K_2$ as the only block and where $\chi =0$ and $\rho =1$; as discussed earlier this produces random recursive trees. In this case, $f(1) = 1$, $g(1) = 1$, and $\lambda_1 = 1$, and so for any positive integer $r$ the vector $\nu = (\nu_1, \ldots, \nu_r)$ defined in (\ref{eq:vhook}) is given by 
\[ \nu = \left(\frac{1}{2}, \frac{1}{4}, \frac{1}{8}, \ldots, \frac{1}{2^r}\right).\]
We see that Theorem \ref{thm:hooking} extends previous results on random recursive trees \cite{MASM:92,JANS:05}. 

More generally, suppose that we look at a preferential attachment tree, where the latch $v$ is chosen with probability proportional to $\chi\deg{v} + \rho$. We once again have $f(1) = 1$ and $g(1) = 1$, and we have that $\lambda_1 = w_1 + \chi = w_2$. We see that $\nu_1 = 1/(w_2 + w_1)$ and by following the recursion of (\ref{eq:vhookdef}) we see that for any $i=2,3, \ldots, $ $\nu_i$ is given by 
\begin{equation}\label{eq:preftree}
\nu_i = \frac{w_{i-1}}{w_2 + w_i}\left(\prod_{j=2}^{i-1}\frac{w_{j-1}}{w_2 + w_j}\right)\cdot \nu_1 = \frac{1}{w_2 + w_1}\prod_{j=2}^{i}\frac{w_{j-1}}{w_2 + w_j}.
\end{equation} 
In particular when $\chi =1$ and $\rho =0$, then $n\lambda_1\nu_i =\frac{4n}{i(i+1)(i+2)}$, and so we see that Theorem \ref{thm:hooking} extends previous results on random plane-oriented recursive trees \cite{MASS:93,JANS:05}, while (\ref{eq:preftree}) along with Theorem \ref{thm:hooking} is the result stated in \cite[Theorem 12.2]{HOJS:17}. 

\begin{remark}
In the literature on random recursive trees and preferential attachment trees, the choice of the latch is usually made proportionally to $\chi\deg^+(v) + \rho'$, where $\deg^+(v)$ is the number of children of $v$. But we can simply let $\rho = \rho' - \chi$ to get the same model, and replace $w_k$ with $w_{k-1}' = \chi(k-1) + \rho'$ so that (\ref{eq:preftree}) resembles more the statements of the previous results \cite{MASM:92, MASS:93, JANS:05, HOJS:17}. The only vertex where this does not translate is the root (or master hook) of the network, since $\deg(H) = \deg^+(H)$ in this case, but see Remarks \ref{rem:naught} and \ref{rem:naughtagain} below for why this does not affect the limiting distribution. 
\end{remark}

\subsubsection{Main results for bipolar networks}\label{sec:bipolarthm}

Let $\mathcal{C} = \{B_1, B_2, \ldots, B_m\}$ be a set of blocks each with a north pole $N_i$ and a south pole $S_i$ identified, and let $\mathcal{B}_0, \mathcal{B}_1, \mathcal{B}_2, \ldots$ be a sequence of bipolar networks grown from $\mathcal{C}$, with the master source labelled $N$ and the master sink labelled $S$. The latches $v$, arcs $(v,u)$, and blocks $B_i$ are chosen in the manner laid out in Section \ref{sec:networks} (by linear preferential attachment with parameters $\chi$ and $\rho$ for the latch, uniformly at random amongst arcs leading out of $v$ for $(v,u)$, and according to its probability $p_i$ for $B_i$). For a positive integer $r$, let 
\[ k_1 < k_2 < \cdots < k_r \]
be the first $r$ essential outdegrees. We introduce similar notations as for the hooking network case. Again, recall that for a positive integer $k$, we let $w_k = \chi k + \rho$. For each block $B_i \in \mathcal{C}$, let $V(B_i)$ be its vertex set. For a positive integer $k$, define 
\begin{equation}\label{eq:fbip}
 f(k) := \sum_{B_i \in \mathcal{C}} p_i \cdot|\{ v \in V(B_i) \setminus \{N_i, S_i\} : \deg^+(v) = k\}|\end{equation}
and for a nonnegative integer $k$, define 
\begin{equation}\label{eq:gbip}
 g(k) := \sum_{\substack{B_i \in \mathcal{C} \\ \deg^+(N_i) = k+1}}\!\!\!\!\!\!\!\!\!p_i.
 \end{equation}
The value $f(k)$ is the expected number of new vertices of outdegree $k$ added at any step, and $g(k)$ is the probability that the outdegree of a latch $v$ is increased by $k$ when $(v,u)$ is replaced with a block (note here that $g(0) \neq 0$ if there is a block whose north pole has outdegree 1). For the blocks of Figure \ref{fig:bipolarseeds} we have that $f(1) = 1$, $f(2) = 1$, and $f(3) = 1/2$, while $g(0) = 1/2$ and $g(1) = 1/2$. 
For a set of blocks $\mathcal{C}$, define 
\begin{equation}\label{eq:lambdabip}
\lambda_1 :=  \sum_{k \geq 1} \left(w_kf(k) + \chi k g(k)\right).
\end{equation}
The value $\lambda_1$ is the expected change in the denominator of (\ref{eq:biplatchprob}) at each step in the growth of the bipolar network. For our running example of bipolar networks grown from the blocks in Figure \ref{fig:bipolarseeds}, if we let $\chi =0$ and $\rho =1$, then 
\begin{equation}\label{eq:lambdabipolarexample}
\lambda_1 = \frac{5}{2}.
\end{equation}
Let $\psi_1 := f(k_1)/(\lambda_1 + w_{k_1}(1 - g(0)))$, and define recursively for $i=2,\ldots,r$
\begin{equation}\label{eq:vbipolardef}
 \psi_i := \frac{1}{\lambda_1 + w_{k_i}(1 - g(0))}\left( f(k_i) + \sum_{j=1}^{i-1}w_{k_j}g(k_i - k_j)\psi_j\right).
 \end{equation}
 The value $\lambda_1\psi_i$ is the limit of the expected proportion of vertices with outdegree $k_i$ (see Remark \ref{rem:bipprop} below). Define 
\begin{equation}\label{eq:vbipolar}
\psi := (\psi_1,\psi_2, \ldots,\psi_r ).
\end{equation}
For our running example of bipolar networks grown from the blocks in Figure \ref{fig:bipolarseeds} with $\chi =0$ and $\rho =1$, and if we let $r =3$, then the first 3 essential outdegrees are $1,2,3$, and 
\begin{equation}\label{eq:vbipolarexample}
\psi = \left(\frac{1}{3}, \frac{7}{18}, \frac{25}{108}\right).
\end{equation}

We have the following multivariate normal limit law for the outdegrees in the growth of bipolar networks.

\begin{theorem}\label{thm:bipolar}
Let $\mathcal{Y}_n = (Y_{n,1}, Y_{n,2}, \ldots, Y_{n,r})$, where $Y_{n,i}$ is the number of vertices with outdegree $k_i$ in $\mathcal{B}_n$, where $\mathcal{B}_n$ is a bipolar network grown from the set of blocks $\mathcal{C}$ using linear preferential attachment with parameters $\chi$ and $\rho$. Let $\lambda_1$ be defined as in (\ref{eq:lambdabip}) and let $\psi$ be the vector defined in (\ref{eq:vbipolardef}) and (\ref{eq:vbipolar}).
Then 
\begin{equation}\label{eq:bipolarconv}
n^{-1/2}(\mathcal{Y}_n - n\lambda_1\psi) \xrightarrow{d} \mathcal{N}(0,\Sigma)
\end{equation}
for some covariance matrix $\Sigma$. 
\end{theorem}

\begin{remark}\label{rem:bipprop}
With the same reasoning as in Remark \ref{rem:hookprop}, we have a weak law of large numbers 
\begin{equation}\label{eq:LLNbip}
 \mathcal{Y}_n/n \xrightarrow{p} \lambda_1\psi
 \end{equation}
and a convergence of the means
\[ \mathbb{E}\mathcal{Y}_n/n \rightarrow \lambda_1 \psi.\]
The convergence in (\ref{eq:LLNbip}) also holds almost surely (see Remark \ref{rem:asurn}).
\end{remark}

Once again, we can say something more about the convergence in (\ref{eq:bipolarconv}) in certain cases.  For each block $B_i$, let $E(B_i)$ be the set of arcs of $B_i$, and let
\begin{equation}\label{eq:sforbips}
 s_i := \sum_{u \in V(B_i)}\!\!\!(\chi\deg^+(u) + \rho) - \chi - \rho = \chi(|E(B_i)|-1) + \rho(|V(B_i)| - 1). 
 \end{equation} 

\begin{corollary}\label{thm:bipolarspecial}
Let $\mathcal{Y}_n = (Y_{n,1}, Y_{n,2}, \ldots, Y_{n,r}),$ where $Y_{n,i}$ is the number of vertices with essential outdegree $k_i$ in $\mathcal{B}_n$, where $\mathcal{B}_n$ is a bipolar network grown from the set of blocks $\mathcal{C}$ using linear preferential attachment with parameters $\chi$ and $\rho$. Let $\lambda_1$ be defined as in (\ref{eq:lambdabip}), let $\psi$ be the vector defined in (\ref{eq:vbipolardef}) and (\ref{eq:vbipolar}), and let $s_i$ be defined as in (\ref{eq:sforbips}) for each block $B_i$. Suppose that there exists a constant $s$ so that $s_i = s$ for all blocks $B_i$. Then the convergence (\ref{eq:bipolarconv}) holds in all moments. In particular, $n^{-1/2}(\mathbb{E}Y_n - n\lambda_1\psi) \rightarrow 0$, and so $n\lambda_1\psi$ in (\ref{eq:bipolarconv}) can be replaced by $\mathbb{E}Y_n$.
\end{corollary}


\begin{remark} We could choose to study the indegrees of bipolar networks instead. Consider networks $\mathcal{B}_0', \mathcal{B}_1', \mathcal{B}_2', \ldots$ grown from the blocks $\mathcal{C} = \{B_1', \ldots, B_m'\}$. Now we choose the latch $v$ proportionally to $\chi \deg^-(v) + \rho$ (instead of $\chi \deg^+(v) + \rho$), and the arc to be replaced with a block is chosen uniformly at random amongst the arcs leading into $v$ (instead of leading out of $v$). The multivariate normal limit law for the indegree distribution of such networks is the same as that for the outdegree distribution of bipolar networks $\mathcal{B}_0, \mathcal{B}_1, \mathcal{B}_2, \ldots$ grown in the manner laid out in Section \ref{sec:bipolar} from the blocks $\mathcal{C} = \{B_1, \ldots, B_m\}$, where the arcs of $B_i'$ are reversed to make $B_i$. 
\end{remark}

\section{P\'olya urns}\label{sec:polya}

A generalized P\'olya urn process $(X_n)_{n=0}^{\infty}$ is defined as follows. There are $q$ types (or colours) $1, 2 \ldots, q$ of balls and for each vector $X_{n} = (X_{n,1}, X_{n,2}, \ldots, X_{n,q})$, the entry $X_{n,i} \geq 0$ is the number of balls of type $i$ in the urn at time $n$, starting with a given (random or not) vector $X_0$. Each type $i$ is assigned an activity $a_i \in \mathbb{R}_{\geq 0}$ and a random vector $\xi_i = (\xi_{i,1}, \xi_{i,2}, \ldots, \xi_{i,q})$ satisfying $\xi_{i,j} \geq 0$ for $i \neq j$ and $\xi_{i,i} \geq -1$. At each time $n \geq 1$, a ball is drawn at random so that the probability of choosing a ball of type $i$ is 
\[ \frac{a_iX_{n-1,i}}{\sum_{j=1}^q a_j X_{n-1,j}}.\]
If the drawn ball is of type $i$ it is replaced along with $\Delta X_{n,j}$ balls of type $j$ for each $j=1, \ldots,q$, where the vector $\Delta X_{n} = (\Delta X_{n,1}, \Delta X_{n,2}, \ldots, \Delta X_{n,q})$ has the same distribution as $\xi_i$ and is independent of everything else that has happened so far. We allow for $\Delta X_{n,i} = -1$, in which case the drawn ball is not replaced. 

The {\em intensity matrix} of the P\'olya urn is the $q \times q$ matrix 
\[ A:= \left(a_j\mathbb{E}\xi_{j,i}\right)_{i,j = 1}^q.\]
By the choice of $\xi_{i,j}$, the matrix $\alpha I + A$ has non-negative entries for a large enough $\alpha$, and so by the standard Perron-Frobenius theory, $A$ has a real eigenvalue $\lambda_1$ such that all other eigenvalues $\lambda \neq \lambda_1$ satisfy $\text{Re}\lambda < \lambda_1$. 

The following assumptions (A1)--(A7) are used in \cite{JANS:04}. In the interpretation of balls in an urn, the random vectors $\xi_i$ and $\Delta X_n$ are integer-valued. However, for our applications, this is not necessarily the case, which is why our assumption (A1) below takes a slightly different form from the standard assumption (A1) in \cite{JANS:04}, taking instead the form discussed in \cite[Remark 4.2]{JANS:04} (note the indices of the variables in (A1) below). A type $i$ is called {\em dominating} if in an urn starting with a single ball of type $i$, there is a positive probability that a ball of type $j$ can be found in the urn at some time for every other type $j$. If every type is dominating, then the urn and its intensity matrix $A$ are {\em irreducible}. 
\begin{enumerate}
\item[(A1)] For each $i$, either 
\begin{enumerate}
\item there is a real number $d_i > 0$ such that $X_{0,i}$ and $\xi_{1,i}, \xi_{2,i}, \ldots, \xi_{q,i}$ are multiplies of $d_i$ and $\xi_{i,i} \geq -d_i$, or 
\item $\xi_{i,i} \geq 0$.
\end{enumerate}
\item[(A2)] $\mathbb{E}(\xi_{i,j}^2) < \infty$ for all $i,j \in \{1,2,\ldots,q\}$.
\item[(A3)] The largest eigenvalue $\lambda_1$ of $A$ is positive. 
\item[(A4)] The largest eigenvalue $\lambda_1$ of $A$ is simple. 
\item[(A5)] There exists a dominating type $i$ with $X_{0,i} >0$. 
\item[(A6)] $\lambda_1$ is an eigenvalue of the submatrix of $A$ given by the dominating types.
\item[(A7)] At each time $n \geq 1$, there exists a ball of dominating type. 
\end{enumerate}

In the P\'olya urns we use, it is obvious that (A1) and (A2) hold. Our intensity matrices are also irreducible, and so (A5) and (A6) hold trivially, while the Perron-Frobenius theorem along with irreducibility guarantee that (A3) and (A4) hold. Our urns always have balls of positive activity, and so (A7) holds by the irreducibility of the urns. 

Denote column vectors as $v$ with $v'$ as its transpose. The transpose of a matrix $A$ is also denoted as $A'$. Let $a = (a_1, \ldots, a_q)'$ denote the vector of activities, and let $u_1'$ and $v_1$ be the left and right eigenvectors of $A$ corresponding to the eigenvalue $\lambda_1$ normalized so that $a \cdot v_1 = a'v_1 = v_1'a = 1$ and $u_1 \cdot v_1 = u_1'v_1 = v_1'u_1 = 1$. Define $P_{\lambda_1} = v_1 u_1'$ and $P_I = I_q - P_{\lambda_1}$. Define the matrices 
\[B_i := \mathbb{E}(\xi_i \xi_i')\]
for every $i = 1, \ldots, q$, denote $v_1 = (v_{1,1}, v_{1,2}, \ldots, v_{1,q})'$, and define the matrix 
\begin{equation}\label{eq:B}
B := \sum_{i=1}^q v_{1,i}a_iB_i.
\end{equation}
In the case where $\text{Re} \lambda < \lambda_1/2$ for every eigenvalue $\lambda \neq \lambda_1$, define 
\begin{equation}\label{eq:sigma1}
 \Sigma_I := \int_{0}^{\infty} P_I e^{sA}B e^{sA'}P_I'e^{-\lambda_1s}ds,
 \end{equation}
where $e^{tA} = \sum_{j=0}^{\infty}t^jA^j/j!$. The result we use from \cite{JANS:04} guarantees that if (A1)--(A7) hold and $\text{Re}\lambda <~ \lambda_1/2$ for all eigenvalues $\lambda \neq \lambda_1$, then $n^{-1/2}(X_n - n\mu) \xrightarrow{d} \mathcal{N}(0,\Sigma)$ for some $\mu = (\mu_1, \ldots, \mu_q)$ and $\Sigma = (\sigma_{i,j})_{i,j=1}^q$. We state below results from \cite{JANS:04} (and gathered in \cite[Theorem 4.1]{HOJS:17}) which give the conditions for convergence to multivariate normal distributions as well as the values of $\mu$ and $\Sigma$. 


\begin{theorem}[{\cite[Theorem 3.22 and Lemmas 5.4 and 5.3(i)]{JANS:04}}]\label{thm:urns}
Assume (A1)--(A7) and that the right and left eigenvectors corresponding to $\lambda_1$ are normalized as above. Assume that $\text{Re} \lambda < \lambda_1/2$ for each eigenvalue $\lambda \neq \lambda_1$. 
\begin{enumerate}
\item[(i)] Then, as $n \to \infty$, 
\begin{equation}\label{eq:conv}
n^{-1/2}(X_n - n\mu) \xrightarrow{d} \mathcal{N}(0, \Sigma)
\end{equation}
with $\mu = \lambda_1v_1$ and some covariance matrix $\Sigma$.
\item[(ii)] Suppose further that, for some $c >0$, 
\[a \cdot \mathbb{E}(\xi_i) = c\]
for every $i=1,\ldots,q$. Then the covariance matrix is given by $\Sigma = c\Sigma_I$, where $\Sigma_I$ is defined in (\ref{eq:sigma1}). 
\item[(iii)] Suppose that (ii) holds and that the matrix $A$ is diagonalizable, and let $\{u_i'\}_{i=1}^q$ and $\{v_i'\}_{i=1}^q$ be dual bases of left and right eigenvectors respectively, i.e., $u_i'A = \lambda_iu_i'$, $Av_i = \lambda_iv_i$, and $u_i'v_j = \delta_{i,j}$. Then the covariance matrix $\Sigma$ is given by 
\begin{equation}\label{eq:sigmadiag}
 \Sigma = c\sum_{j,k=2}^q \frac{u_j'Bu_k}{\lambda_1 - \lambda_j - \lambda_k}v_jv_k',
 \end{equation}
where $B$ is defined in (\ref{eq:B}).
\end{enumerate}
\end{theorem}


\begin{remark}\label{rem:naught}
So long as (A5) is satisfied, the initial configuration $X_0$ of the urn does not have any effects on the limiting distribution. 
\end{remark}


\begin{remark}\label{rem:asurn}
From (\ref{eq:conv}), we get a a weak law of large numbers 
\begin{equation}\label{eq:LLNurn}
\mathbb{E}X_n/n \xrightarrow{p} \mu.
\end{equation}
In fact, the convergence in (\ref{eq:LLNurn}) holds almost surely for urns satisfying (A1)--(A7) (see \cite[Theorem 3.21]{JANS:04}). Therefore, once the convergences in (\ref{eq:hookconv}) and (\ref{eq:bipolarconv}) are established via P\'olya urns in the following section, the convergences in (\ref{eq:LLNhook}) and (\ref{eq:LLNbip}) hold almost surely.
\end{remark}

\begin{remark}\label{rem:moms}  A recent result by Janson and Pouyanne \cite{JAPO:18} guarantees tha the convergence (\ref{eq:conv}) holds in all moments for certain balanced generalized P\'olya urns; an urn is balanced if the change in total activity at every step is constant. Some of our urns satisfy the conditions of \cite[Theorem 1.1]{JAPO:18} which implies in particular that $n^{-1/2}(\mathbb{E}X_n - n\mu) \rightarrow 0$, and so $n\mu$ in (\ref{eq:conv}) can be replaced by $\mathbb{E}X_n$.
\end{remark}

\section{Proofs}\label{sec:proofs}

We start by setting up P\'olya urns so that balls in the urn correspond to vertices in the growth of our network. Next, we prove important properties of the intensity matrices associated with these P\'olya urns. Finally, the pieces are placed together to prove our main results.

\subsection{Vertices as balls}\label{sec:vertsasballs}

In this section, we outline how we use the evolution of generalized P\'olya urns to describe the evolutions of the degree distributions in the networks that we study. Throughout the section the notation {\em (out)degree} is used so that the discussion applies to both types of networks simultaneously. Recall that Theorem \ref{thm:hooking} and Corollary \ref{thm:hookingspecial} apply to degrees of hooking netwoks, while Theorem \ref{thm:bipolar} and Corollary \ref{thm:bipolarspecial} apply to outdegrees of bipolar networks. 

We start by first looking at an urn with infinitely many types. We assign a type to each (out)degree in the network so that a ball of type $k$ represents a vertex of (out)degree $k$. We initiate each network by choosing a block from the list of blocks. This corresponds to starting a P\'olya urn with a ball of the matching type for the (out)degree of each vertex in the block. In the evolution of the network, when a block is attached, this corresponds to choosing a ball in the urn of type corresponding to the (out)degree of the latch $v$ and replacing it with a ball representing the new (out)degree of $v$ along with balls representing the (out)degrees of the rest of the vertices of the newly attached block. Since a latch of (out)degree $k$ is chosen at random proportionally to $w_k = \chi k + \rho$, then all balls of type $k$ have activity $w_k$ in the P\'olya urn so that a ball of type $k$ is chosen at random proportionally to its activity $w_k$. 

The P\'olya urn described above has infinitely many types, and so Theorem \ref{thm:urns} does not apply. Therefore, we would like to instead use an urn with finitely many types in the same manner as is done in \cite{JANS:05} and \cite{HOJS:17}. The urn is replaced with the following P\'olya urn: let $d$ be a positive integer corresponding to the largest (out)degree we wish to study in this instance of the model. A new ball of special type $\ast$ with activity $a_{\ast} = 1$ is introduced, and for every $k > d$, each ball of type $k$ is replaced with $w_k$ balls of special type $\ast$. In this way, the probability of choosing a ball of special type in the new urn is equal to the probability of choosing a ball of type greater than $d$ in the old urn. If a latch $v$ with (out)degree $k \leq d$ is chosen, and a block is attached so that $v$ now has (out)degree $k + j > d$, then the ball of type $k$ is removed and $w_{k+j}$ balls of special type are added. If instead $v$ has (out)degree $k > d$ and a block is attached so that the (out)degree of the vertex is now $k+j$, then the ball of special type that was chosen is placed back in the urn, along with $\chi j$ balls of special type. 

The final change we will make to our urn is to represent the master hook of the hooking network or the master source of the bipolar network, say with (out)degree $k$, with $w_k$ balls of special type in our urn. This guarantees that all types of balls in the urn that are not special types correspond to (out)degrees that are essential; recall from Definition \ref{def:admissible} that a positive integer $k$ is an essential (out)degree if there is a positive probability that at some point in the growth of the network at least two vertices have (out)degree $k$, and recall from Proposition \ref{pro:adm} that only the master hook of the hooking network or the master source of the bipolar network may have a nonessential degree. For a positive integer $d$, the possible types of balls present in the urn are exactly the essential (out)degrees less than or equal to $d$, together with a ball of special type $\ast$. In our intensity matrix, we can then omit the rows and columns corresponding to types that are never present in the urn. By restricting to essential (out)degrees, it can be verified that now every ball in the urn is of dominating type. No matter the initial network (or initial configuration of the urn), there is a positive probability that a ball representing a vertex with the essential (out)degree $k$ will be present in the urn. Therefore the urn (and its intensity matrix) is irreducible. As discussed in Section \ref{sec:polya}, it is easy to verify that the assumptions (A1)--(A7) are satisfied for irreducible urns. To avoid confusion, we label the type of a ball with the (out)degree of the vertex it represents. 

We illustrate how to calculate the intensity matrices for the urns associated with our running examples of hooking networks and bipolar networks given in Section \ref{sec:networks}. 

\subsubsection{A P\'olya urn for our running example of a hooking network}

Consider the blocks in Figure \ref{fig:hookseeds}, and a sequence of hooking networks grown from these blocks. Let's look at the instance of the model where the choice of a latch is made proportionally to its degree (i.e., when $\chi =1$, $\rho = 0$ and so $w_k = k$). Suppose we look at vertices with degrees less than or equal to 5. As discussed after the definition of essential degrees (Definition \ref{def:admissible}), the essential degrees for these hooking networks are the odd numbers; and so 1, 3, 5 are the essential degrees less than or equal to 5. The images in Figure \ref{fig:hookingnetworkreplace} illustrate the possibilities for replacing a ball of type $k$, corresponding to attaching a block to a latch with degree $k$. The probabilities in the figure are the probabilities $p_i$ assigned to the blocks in Figure \ref{fig:hookseeds}.  

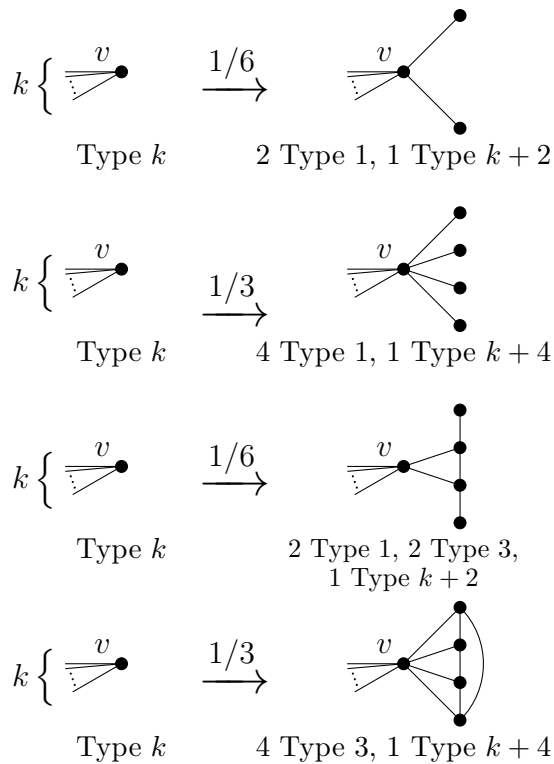
\begin{figure}[h!]
\centering
\begin{tikzpicture}[scale = 0.75]

\node at (3.75,12.75) {$k \begin{cases}  \\ \end{cases}$};
\filldraw(5,13) circle (3pt)  node(v){} node[above left]{$v$};
\node (v2) at (4,13){};
\node (v1) at (4.0038,12.91285){};
\node (v3) at (4.134, 12.5){};
\draw (v.center) -- (v1.center);
\draw (v.center) -- (v2.center);
\draw (v.center) -- (v3.center);
\filldraw (4.1, 12.825) circle (0.25pt);
\filldraw (4.135, 12.725) circle (0.25pt);
\filldraw (4.17, 12.625) circle (0.25pt);

\node at (5, 11.5){\small Type $k$};

\node at (7,13){\Large$\xrightarrow{1/6}$};

\filldraw(10,13) circle (3pt)  node(v){} node[above left]{$v$};
\node (v2) at (9,13){};
\node (v1) at (9.0038,12.91285){};
\node (v3) at (9.134, 12.5){};
\draw (v.center) -- (v1.center);
\draw (v.center) -- (v2.center);
\draw (v.center) -- (v3.center);
\filldraw (9.1, 12.825) circle (0.25pt);
\filldraw (9.135, 12.725) circle (0.25pt);
\filldraw (9.17, 12.625) circle (0.25pt);

\filldraw(11, 14) circle (3pt) node(r){};
\filldraw(11, 12) circle (3pt) node(l){};
\draw (r.center) -- (v.center) -- (l.center);
\node at (10, 11.5){\small 2 Type 1, 1 Type $k+2$ };

\node at (3.75,9.25) {$k \begin{cases}  \\ \end{cases}$};
\filldraw(5,9.5) circle (3pt)  node(v){} node[above left]{$v$};
\node (v2) at (4,9.5){};
\node (v1) at (4.0038,9.41285){};
\node (v3) at (4.134, 9){};
\draw (v.center) -- (v1.center);
\draw (v.center) -- (v2.center);
\draw (v.center) -- (v3.center);
\filldraw (4.1, 9.325) circle (0.25pt);
\filldraw (4.135, 9.225) circle (0.25pt);
\filldraw (4.17, 9.125) circle (0.25pt);

\node at (5, 8){\small Type $k$};

\node at (7,9){\Large$\xrightarrow{1/3}$};

\filldraw(10,9.5) circle (3pt)  node(v){} node[above left]{$v$};
\node (v2) at (9,9.5){};
\node (v1) at (9.0038,9.41285){};
\node (v3) at (9.134, 9){};
\draw (v.center) -- (v1.center);
\draw (v.center) -- (v2.center);
\draw (v.center) -- (v3.center);
\filldraw (9.1, 9.325) circle (0.25pt);
\filldraw (9.135, 9.225) circle (0.25pt);
\filldraw (9.17, 9.125) circle (0.25pt);

\filldraw(11, 10.5) circle (3pt) node(r){};
\filldraw(11,9.8333) circle (3pt) node(br){};
\filldraw(11, 9.1666) circle (3pt) node(bl){};
\filldraw(11,8.5) circle (3pt) node(l){};
\draw (r.center) -- (v.center) -- (l.center);
\draw (br.center) -- (v.center) -- (bl.center);
\node at (10, 8){\small 4 Type 1, 1 Type $k+4$ };

\node at (3.75,5.75) {$k \begin{cases}  \\ \end{cases}$};
\filldraw(5,6) circle (3pt)  node(v){} node[above left]{$v$};
\node (v2) at (4,6){};
\node (v1) at (4.0038,5.91285){};
\node (v3) at (4.134, 5.5){};
\draw (v.center) -- (v1.center);
\draw (v.center) -- (v2.center);
\draw (v.center) -- (v3.center);
\filldraw (4.1, 5.825) circle (0.25pt);
\filldraw (4.135, 5.725) circle (0.25pt);
\filldraw (4.17, 5.625) circle (0.25pt);

\node at (5, 4.5){\small Type $k$};

\node at (7,6){\Large$\xrightarrow{1/6}$};

\filldraw(10,6) circle (3pt)  node(v){} node[above left]{$v$};
\node (v2) at (9,6){};
\node (v1) at (9.0038,5.91285){};
\node (v3) at (9.134, 5.5){};
\draw (v.center) -- (v1.center);
\draw (v.center) -- (v2.center);
\draw (v.center) -- (v3.center);
\filldraw (9.1, 5.825) circle (0.25pt);
\filldraw (9.135, 5.725) circle (0.25pt);
\filldraw (9.17, 5.625) circle (0.25pt);

\filldraw(11, 7) circle (3pt) node(r){};
\filldraw(11,6.3333) circle (3pt) node(br){};
\filldraw(11,5.6666) circle (3pt) node(bl){};
\filldraw(11, 5) circle (3pt) node(l){};
\draw (r.center) -- (l.center);
\draw (br.center) -- (v.center) -- (bl.center);
\node at (10, 4.5){\footnotesize 2 Type 1, 2 Type 3,};
\node at(10,4){\footnotesize  1 Type $k+2$ };

\node at (3.75,2.25) {$k \begin{cases}  \\ \end{cases}$};
\filldraw(5,2.5) circle (3pt)  node(v){} node[above left]{$v$};
\node (v2) at (4,2.5){};
\node (v1) at (4.0038,2.41285){};
\node (v3) at (4.134, 2){};
\draw (v.center) -- (v1.center);
\draw (v.center) -- (v2.center);
\draw (v.center) -- (v3.center);
\filldraw (4.1, 2.325) circle (0.25pt);
\filldraw (4.135, 2.225) circle (0.25pt);
\filldraw (4.17, 2.125) circle (0.25pt);

\node at (5, 1){\small Type $k$};

\node at (7,2.5){\Large$\xrightarrow{1/3}$};

\filldraw(10,2.5) circle (3pt)  node(v){} node[above left]{$v$};
\node (v2) at (9,2.5){};
\node (v1) at (9.0038,2.41285){};
\node (v3) at (9.134, 2){};
\draw (v.center) -- (v1.center);
\draw (v.center) -- (v2.center);
\draw (v.center) -- (v3.center);
\filldraw (9.1, 2.325) circle (0.25pt);
\filldraw (9.135, 2.225) circle (0.25pt);
\filldraw (9.17, 2.125) circle (0.25pt);

\filldraw(11, 3.5) circle (3pt) node(r){};
\filldraw(11, 2.8333) circle (3pt) node(br){};
\filldraw(11, 2.1666) circle (3pt) node(bl){};
\filldraw(11, 1.5) circle (3pt) node(l){};
\draw (br.center) -- (bl.center) -- (v.center) -- (br.center) -- (r.center) -- (v.center) -- (l.center) -- (bl.center);
\draw (l.center) arc (-45:45:1.4142);

\node at (10, 1){\small 4 Type 3, 1 Type $k+4$ };

\end{tikzpicture}
\caption{The replacements of a ball of type $k$ in a hooking network grown from the blocks in Figure \ref{fig:hookseeds} }
\label{fig:hookingnetworkreplace}
\end{figure}

The intensity matrix for this urn has 4 rows and columns: one of each for balls of type 1, 3, 5, and the last row and column for balls of special type $\ast$. Let's consider what happens when a block is attached to a latch with degree 1; this corresponds to choosing a ball of type 1. The probability that the block $G_1$ is attached is $1/6$. The hook of $G_1$ has degree 2 and the two other vertices  have degree 1. The ball of type 1 is removed and replaced with a ball of type 3 (the new degree of the latch $v$) along with two new balls of type 1. Performing similar calculations for the other blocks with the help of Figure \ref{fig:hookingnetworkreplace}, we get that 
\[ \mathbb{E}\xi_1 = \frac{1}{6}\left(\begin{array}{c} 1 \\ 1 \\ 0 \\ 0 \end{array}\right) 
+\frac{1}{3}\left(\begin{array}{c} 3 \\ 0 \\ 1 \\ 0 \end{array}\right)
+\frac{1}{6}\left(\begin{array}{c} 1 \\ 3 \\ 0 \\ 0  \end{array}\right)
+\frac{1}{3}\left(\begin{array}{c} -1 \\ 4 \\ 1 \\ 0  \end{array}\right)
= \frac{1}{6}\left(\begin{array}{c} 6 \\ 12 \\ 4 \\ 0 \end{array}\right). \]
Recall that the rows and columns for nonessential degrees are removed, and so the first row represents balls of type 1, the second row for balls of type 3, the third for balls of type 5, and the final row for balls of special type $\ast$. 

Now consider what happens when a ball of type 3 is chosen, i.e., if a vertex $v$ with degree 3 is chosen as a latch. If a hook with degree 4 is attached to $v$, the degree of $v$ is increased to 7. Recall that we instead place $w_7 = 7$ balls of special type when this happens. Performing similar calculations as above with the help of Figure \ref{fig:hookingnetworkreplace} yields
\[ \mathbb{E}\xi_3 =  \frac{1}{6}\left(\begin{array}{c} 2 \\ -1 \\ 1 \\ 0 \end{array}\right) 
+\frac{1}{3}\left(\begin{array}{c} 4 \\ -1 \\ 0 \\ 7 \end{array}\right)
+\frac{1}{6}\left(\begin{array}{c} 2 \\ 1 \\ 1 \\ 0 \end{array}\right)
+\frac{1}{3}\left(\begin{array}{c} 0 \\ 3 \\ 0 \\ 7  \end{array}\right)
= \frac{1}{6}\left(\begin{array}{c} 12 \\ 4 \\ 2 \\ 28 \end{array}\right). \]
Performing similar calculations when a ball of type 5 is chosen gives
\[ \mathbb{E}\xi_5 = \frac{1}{6}\left(\begin{array}{c} 2 \\ 0 \\ -1 \\ 7 \end{array}\right) 
+\frac{1}{3}\left(\begin{array}{c} 4 \\ 0 \\ -1 \\ 9 \end{array}\right)
+\frac{1}{6}\left(\begin{array}{c} 2 \\ 2 \\ -1 \\ 7 \end{array}\right)
+\frac{1}{3}\left(\begin{array}{c} 0 \\ 4 \\ -1 \\ 9  \end{array}\right)
= \frac{1}{6}\left(\begin{array}{c} 12 \\ 10 \\ -6 \\ 50 \end{array}\right). \]
Finally let's consider attaching a block to a vertex of degree greater than 5, or to the master hook of the network. In either case, this corresponds to choosing a ball of special type. If the hook of the block $G_i$ attached has degree two, then the ball of special type is replaced along with another $2\chi = 2$ balls of special type, while $4\chi = 4$ balls of special type are added if the hook has degree 4. Therefore, we calculate for the special type $\ast$
\[\mathbb{E}\xi_{\ast} =  \frac{1}{6}\left(\begin{array}{c} 2 \\ 0 \\ 0 \\  2 \end{array}\right) 
+\frac{1}{3}\left(\begin{array}{c} 4 \\ 0 \\ 0 \\  4 \end{array}\right)
+\frac{1}{6}\left(\begin{array}{c} 2 \\ 2 \\ 0 \\ 2 \end{array}\right)
+\frac{1}{3}\left(\begin{array}{c} 0 \\ 4 \\ 0 \\  4 \end{array}\right)
= \frac{1}{6}\left(\begin{array}{c} 12 \\ 10 \\ 0 \\  20 \end{array}\right). \]
The activities for the types are $w_1 = 1, w_3 = 3$ and $w_5 = 5$ for types $1,3,5$ respectively, while the special type $\ast$ has activity 1 (as discussed earlier). The intensity matrix $A$ consists of $\mathbb{E}\xi_1, 3\mathbb{E}\xi_3, 5\mathbb{E}\xi_5$ for the first 3 columns, and $\mathbb{E}\xi_{\ast}$ for the last column, thus we get  
\[ A = \frac{1}{6}\left(\begin{array}{cccc}
6 & 36 & 60  & 12 \\
12 & 12 & 50  & 10 \\
4 & 6 & -30 &  0 \\
0 & 84 & 250 & 20 \\
\end{array} \right).\]
One can verify that the eigenvalues of $A$ are $\lambda_1 = 31/3$ and $-1,-3, -5$ and we see that $\lambda_1$ is what was calculated in (\ref{eq:lambdahookexample}). By Theorem \ref{thm:urns}, we have a multivariate normal limit law. One can also verify that the right eigenvector $v_1$ of $A$ associated with $\lambda_1$ satisfying $a \cdot v_1 = 1$, where $a = (1,3, 5,1)$ is the vector of activities, is 
\[ v_1 = \left( \frac{6}{34}, \frac{11}{85}, \frac{63}{3910}, \frac{1387}{3910}\right)'.\]
Restricted to the first 3 entries, the vector $v_1$ is exactly the vector $\nu$ calculated in (\ref{eq:vhookexample}), and so by Theorem \ref{thm:urns}, Theorem \ref{thm:hooking} is true in this particular case. 

\subsubsection{A P\'olya urn for our running example of a bipolar network}

Now consider the blocks of Figure \ref{fig:bipolarseeds} and a sequence of bipolar networks grown from these blocks. Let's look at the instance of the model where the choice of the latch is made uniformly at random (i.e., when $\chi = 0$, $\rho = 1$, and so $w_k =1$). All positive integers are essential outdegrees. The images of Figure \ref{fig:bipolarnetworkreplace} illustrate the possibilities of replacing a ball of type $k$, corresponding to choosing a latch $v$ with outdegree $k$ and one of the arcs leading out of $v$ uniformly at random. The probabilities in the figure are the probabilities $p_i$ assigned to the blocks in Figure \ref{fig:bipolarseeds}. 

\begin{figure}[h!]
\centering
\begin{tikzpicture}[scale = 0.75]

\node at (3.6,12.825) {$k-1 \begin{cases}  \\ \\  \end{cases}$};
\filldraw(5,13) circle (3pt)  node(v){} node[above left]{$v$};
\node (v2) at (4.02,13.19715){};
\node (v1) at (4.0038,12.91285){};
\node (v3) at (4.134, 12.5){};
\filldraw(6.5,13) circle (3pt)  node(u){} node[above right]{$u$};
\node (uv) at (6, 13){};
\draw[->] (v.center) -- (v1.center);
\draw[->] (v.center) -- (v2.center);
\draw[->] (v.center) -- (v3.center);
\draw[->] (v.center) -- (uv);
\draw (v.center) -- (u.center);
\filldraw (4.1, 12.825) circle (0.25pt);
\filldraw (4.135, 12.725) circle (0.25pt);
\filldraw (4.17, 12.625) circle (0.25pt);

\node at (5, 11.5){\small Type $k$};

\node at (7.75,13){\Large$\xrightarrow{1/2}$};

\filldraw(10,13) circle (3pt)  node(v){}node[above left]{$v$};
\node (v2) at (9.02,13.19715){};
\node (v1) at (9.0038,12.91285){};
\node (v3) at (9.134, 12.5){};
\draw[->] (v.center) -- (v1.center);
\draw[->] (v.center) -- (v2.center);
\draw[->] (v.center) -- (v3.center);
\filldraw (9.1, 12.825) circle (0.25pt);
\filldraw (9.135, 12.725) circle (0.25pt);
\filldraw (9.17, 12.625) circle (0.25pt);

\filldraw(10,13) circle (3pt) node(n1){};
\node(nb1) at (12.5, 12.5){};
\filldraw(13, 12) circle (3pt) node (b1){};
\node(nt1) at (12.5, 13.5){};
\filldraw(13, 14) circle (3pt) node(t1){};
\node(bm1) at (13.5, 12.5){};
\node(nm1) at (11,13){};
\filldraw(12,13) circle (3pt) node (m1){};
\node(tm1) at (13.5, 13.5){};
\filldraw(14,13) circle (3pt) node (s1){} node[above right]{$u$};
\node(ms1) at (13,13){};
\draw[->] (m1.center) -- (nb1.center);
\draw[->] (m1.center) -- (nt1.center);
\draw[->] (n1.center) -- (nm1.center);
\draw[->] (b1.center) -- (bm1.center);
\draw[->] (t1.center) -- (tm1.center);
\draw[->] (m1.center) -- (ms1.center);
\draw (n1.center) -- (m1.center) -- (s1.center);
\draw (m1.center) -- (t1.center) -- (s1.center);
\draw (m1.center) -- (b1.center) -- (s1.center);

\node at (12, 11.5){\small 2 Type 1, 1 Type 3, 1 Type $k$};

\node at (3.6,9.325) {$k-1 \begin{cases}  \\ \\  \end{cases}$};
\filldraw(5,9.5) circle (3pt)  node(v){} node[above left]{$v$};
\node (v2) at (4.02,9.69715){};
\node (v1) at (4.0038,9.41285){};
\node (v3) at (4.134, 9){};
\filldraw(6.5,9.5) circle (3pt)  node(u){} node[above right]{$u$};
\node (uv) at (6, 9.5){};
\draw[->] (v.center) -- (v1.center);
\draw[->] (v.center) -- (v2.center);
\draw[->] (v.center) -- (v3.center);
\draw[->] (v.center) -- (uv);
\draw (v.center) -- (u.center);
\filldraw (4.1, 9.325) circle (0.25pt);
\filldraw (4.135, 9.225) circle (0.25pt);
\filldraw (4.17, 9.125) circle (0.25pt);

\node at (5, 8){\small Type $k$};

\node at (7.75,9.5){\Large$\xrightarrow{1/2}$};

\filldraw(10,9.5) circle (3pt)  node(v){}node[above left]{$v$};
\node (v2) at (9.02,9.69715){};
\node (v1) at (9.0038,9.41285){};
\node (v3) at (9.134, 9){};
\draw[->] (v.center) -- (v1.center);
\draw[->] (v.center) -- (v2.center);
\draw[->] (v.center) -- (v3.center);
\filldraw (9.1, 9.325) circle (0.25pt);
\filldraw (9.135, 9.225) circle (0.25pt);
\filldraw (9.17, 9.125) circle (0.25pt);

\filldraw(10,9.5) circle (3pt) node (n2){};
\node(nt2) at (11,10){};
\filldraw(12,10.5) circle (3pt) node (t2){};
\node(nb2) at (11,9){};
\filldraw(12,8.5) circle (3pt) node (b2){};
\node(ts2) at (13,10){};
\node(bs2) at (13,9){};
\filldraw(14,9.5) circle (3pt) node (s2){} node[above right]{$u$};
\draw[->] (n2.center) -- (nt2.center);
\draw[->] (n2.center) -- (nb2.center);
\draw[->] (t2.center) -- (ts2.center);
\draw[->] (b2.center) -- (bs2.center);
\draw[->] (t2.center) arc (135:182.5:1.4142);
\draw[->] (b2.center) arc (-45:2.5:1.4142);
\draw (n2.center) -- (t2.center) -- (s2.center);
\draw (n2.center) -- (b2.center) -- (s2.center);
\draw (t2.center) arc (135: 225: 1.4142);
\draw (b2.center) arc (-45: 45:1.4142);

\node at (12, 8){\small 2 Type 2, 1 Type $k+1$};
\end{tikzpicture}
\caption{The replacement of a ball of type $k$ in a bipolar network grown from the blocks in Figure \ref{fig:bipolarseeds}}
\label{fig:bipolarnetworkreplace}
\end{figure}
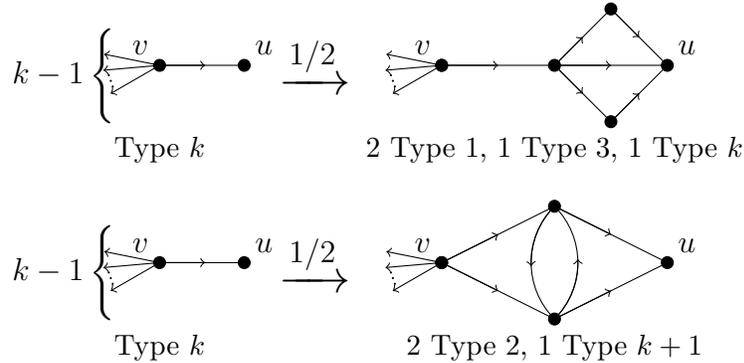

Suppose we look at vertices with outdegrees less than or equal to 3. We can calculate the intensity matrix in the same way as the intensity matrix for the hooking network example above. The main difference in this case is that there is a positive probability that the outdegree of a latch $v$ is not changed. For example, if a ball of type $2$ is chosen; that is, if a latch $v$ with outdegree 2 is chosen, then with probability $1/2$, the degree of $v$ is not changed after the block $B_1$ is attached. In this case, the ball of type 2 is replaced in the urn, along with 2 balls of type $1$ and one ball of type $3$. We can calculate 
\[ \mathbb{E}\xi_{2} = \frac{1}{2}\left(\begin{array}{c} 2 \\ 0 \\ 1 \\ 0 \end{array}\right) 
+ \frac{1}{2}\left(\begin{array}{c} 0\\ 1 \\ 1 \\ 0 \end{array}\right) 
= \frac{1}{2} \left(\begin{array}{c} 2 \\ 1 \\ 2\\ 0 \end{array}\right).\]
For the urn in this case, a vertex with outdegree greater than 3 is represented by a single ball of special type $\ast$. The intensity matrix is 
\begin{equation}\label{eq:Abip}
 A = \frac{1}{2} \left( \begin{array}{cccc} 
1 & 2 & 2 & 2 \\
3 & 1 & 2 & 2 \\
1 & 2 & 0 & 1 \\
0 & 0 & 1 & 0 \end{array}\right).
\end{equation}
The eigenvalues for $A$ are $\lambda_1 = 5/2$ and $-1/2$, and we see that $\lambda_1$ is precisely what was calculated in (\ref{eq:lambdabipolarexample}). The right eigenvector $v_1$ of $A$ associated with $\lambda_1$ whose entries sum to 1 is 
\[ v_1 =  \left(\frac{1}{3}, \frac{7}{18}, \frac{25}{108}, \frac{5}{108}\right)'. \]
Restricted to the first $3$ entries, the vector $v_1$ is exactly the vector $\psi$ calculated in (\ref{eq:vbipolarexample}). The multivariate normal limit law claimed by Theorem \ref{thm:bipolar} holds by Theorem \ref{thm:urns} in this case.


\subsection{Properties of the intensity matrices}\label{sec:A}
Recall that $w_k = \chi k + \rho$.
Let $A = (a_{ij})_{i,j = 1}^{r+1}$ be the $(r+1) \times (r+1)$ matrix with entries 
\begin{equation}\label{eq:A}
a_{ij} = \begin{cases}
w_{k_j}f(k_i) & i < j \leq r  \\
f(k_i) & i < j = r+1 \\
w_{k_i}(f(k_i) + g(0) - 1) & i=j \leq r \\
w_{k_j}(f(k_i) + g(k_i - k_j)) & j < i \leq r \\
w_{k_j}\sum_{k>k_r}w_k(f(k) + g(k-k_j))  & j < i = r+1 \\
\sum_{k>k_r}w_k f(k)  + \sum_{k \geq 1}\chi k g(k) & j=i = r+1.
\end{cases}
\end{equation}
where $f(k)$ was introduced in (\ref{eq:fhook}) and (\ref{eq:fbip}), $g(k)$ was introduced in (\ref{eq:ghook}) and (\ref{eq:gbip}), and $k_1, \ldots, k_r$ are essential degrees. We prove properties of $A$ that are useful to the proofs of our main results. From Theorem \ref{thm:urns}, we see that to prove our main result, we need to prove properties of the eigenvalues and eigenvectors of $A$. 
The eigenvalues and eigenvectors of $A$ depend on properties of the values $f(k)$ and $g(k)$. These properties are gathered in the following proposition. 
\begin{proposition}\label{lem:ass}
For $f(k)$ defined in (\ref{eq:fhook}) and (\ref{eq:fbip}), and $g(k)$ defined in (\ref{eq:ghook}) and (\ref{eq:gbip}), the following properties hold:
\begin{itemize}
\item[(F)  ] If $k \leq k_r$ and $k \neq k_{i}$ for all $i=1, \ldots, r$, then $f(k) = 0$. 
\item[(G1)] $\sum_{k \geq 0} g(k) = 1$. 
\item[(G2)] If $k \leq k_r$ and $k \neq k_{i}$ for all $i=1, \ldots, r$, then $g(k-k_j) = 0$ for all $j = 1, \ldots, r$.   
\end{itemize}
\end{proposition}
\begin{proof}
In the interest of space, the lemma is proved for both hooking networks and bipolar networks simultaneously. The notation {\em (out)degree} is used, and is interpreted as {\em degree} for hooking networks and {\em outdegree} for bipolar networks. 

If $f(k) \neq 0$, then there is a positive probability that at any step in the growth of the network, a new vertex (that is not the master hook or the master source) appears with (out)degree $k$. By Definition \ref{def:admissible} and by Proposition \ref{pro:adm}, $k$ is an essential (out)degree in this case, and so if $k \leq k_r$, then $k \in \{k_1, \ldots, k_r\}$, proving that (F) holds. The property (G1) holds since $\sum_{k \geq 0} g(k) = p_1 + \cdots + p_m = 1$, where $p_i$ is the probability of the block $G_i$ or $B_i$. As for the property (G2), assume that $g(k - k_j) \neq 0$ for some essential (out)degree $k_j \leq k_r$. Since $k_j$ is an essential (out)degree, there is a positive probability that some vertex $v$ (that is not the master hook or the master source) has (out)degree $k_j$. By definition, there is a probability of $g(k-k_j)$  that the (out)degree of $v$ is increased to $k$ if a hook is fused to $v$. Therefore, there is a positive probability that there is a vertex with (out)degree $k$, and so $k$ is an essential (out)degree, again by Definition \ref{def:admissible} and Proposition \ref{pro:adm}. If $k \leq k_r$, then $k \in \{k_1, \ldots, k_r\}$, and so (G2) holds. 
\end{proof}

Let 
\begin{equation*}
\lambda_1 =  \sum_{k \geq 1} \left(w_kf(k) + \chi kg(k)\right)
\end{equation*}
be the value defined in (\ref{eq:lambdahook}) and (\ref{eq:lambdabip}). 
We calculate the eigenvalues of $A$ in the following lemma. 

\begin{lemma}\label{lem:eigenvalues}
The matrix $A$ has eigenvalues
\[ \lambda_1, w_{k_1}(g(0) - 1), w_{k_2}(g(0) - 1), \ldots, w_{k_r}(g(0) - 1).\]
\end{lemma}

\begin{proof}
We can calculate the eigenvalues of $A$ directly. For any $\lambda$, look at the matrix $A - \lambda I$. For each $i=1, \ldots r$, add $w_{k_i}$ times row $i$ to row $r+1$ of $A - \lambda I$ to get the matrix $A'_{\lambda}$. Using properties (F) and (G2), along the $(r+1)$-th row of $A'_{\lambda}$, the $j$-th entry for $j=1, \ldots, r$ is 
\begin{equation}\label{eq:A'j}
w_{k_j}\left(\sum_{k \geq 1} w_kf(k) + \sum_{k \geq k_j} w_kg(k-k_j) - w_{k_j} - \lambda \right)
\end{equation}
while the $(r+1)$-th entry is 
\begin{equation}\label{eq:A'r+}
\sum_{k\geq1}(w_kf(k) + \chi k g(k)) - \lambda = \lambda_1 - \lambda.
\end{equation}
Next, subtract $w_{k_j}$ times column $r+1$ from column $j$ in $A_{\lambda}'$ for every $j=1,\ldots,r$ to get the matrix $A''_{\lambda}$. Since 
\begin{equation}\label{eq:A''r+}
 w_kg(k-k_j) - \chi(k-k_j) g(k-k_j) =  (\chi k_j + \rho)g(k-k_j) = w_{k_j}g(k-k_j),
 \end{equation}
the $j$-th entry for $j=1, \ldots, r$ of the $(r+1)$-th row is 
\begin{align}
&w_{k_j}\left(\sum_{k \geq 1} w_kf(k) + \sum_{k \geq k_j} w_kg(k-k_j) - w_{k_j} - \lambda \right)- w_{k_j}\left(\lambda_1 - \lambda\right) \nonumber\\
&\hspace{10mm} = w_{k_j}\left(\sum_{k \geq 1} w_kf(k) + \sum_{k \geq k_j} w_kg(k-k_j) - w_{k_j} - \lambda_1 \right) \label{eq:A'rowlater} \\
&\hspace{10mm} =w_{k_j}\left(\sum_{k \geq 1} w_kf(k) + \sum_{k \geq k_j} w_kg(k-k_j) - w_{k_j}  - \sum_{k\geq1}(w_kf(k) + \chi k g(k))\right) \nonumber \\
&\hspace{10mm} =w_{k_j}\sum_{k \geq k_j}w_kg(k-k_j) - w_{k_j}\sum_{k \geq k_j}\chi(k - k_j)g(k - k_j) - w_{k_j}^2 \nonumber \\
&\hspace{10mm} =w_{k_j}\sum_{k \geq k_j} w_{k_j}g(k-k_j) - w_{k_j}^2 \tag*{(by (\ref{eq:A''r+}))} \nonumber  \\
&\hspace{10mm} = w_{k_j}^2 - w_{k_j}^2 = 0. \tag*{(by property (G1))} \nonumber
\end{align}

For every $i,j \leq r$, the $i,j$-th entry of $A''_{\lambda}$ is simply $a_{ij} - w_{k_j}f(k_i)$ when $i\neq j$ and $a_{ii} - \lambda - w_{k_i}f(k_i)$ on the diagonals, where $a_{ij}$ is given in (\ref{eq:A}). 
Therefore, $A''_{\lambda}$ is the following $(r+1)\times (r+1)$ matrix
\[ A''_{\lambda} = \left( \begin{array}{ccccc}
w_{k_1}(g(0) - 1) - \lambda & 0 & \cdots & 0 & f(k_1) \\ \\
w_{k_1}g(k_2 - k_1) & w_{k_2}(g(0) - 1) - \lambda & \cdots & 0 & f(k_2) \\ \\
\vdots & \vdots & \ddots & \vdots & \vdots \\ \\
w_{k_1}g(k_r - k_1) & w_{k_2}g(k_r - k_2) & \cdots & w_{k_r}(g(0) - 1) - \lambda& f(k_r) \\ \\
0 & 0 & \cdots & 0 & \lambda_1 - \lambda
\end{array}\right).
\]
Since the determinant of a matrix is unchanged by adding one row to another or by subtracting a column from another, both $A - \lambda I$ and $A''_{\lambda}$ have the same determinant. We can calculate the determinant of $A''_{\lambda}$ by expanding along the bottom row, and since the upper $r \times r$ matrix of $A''_{\lambda}$ is lower triangular, we see immediately that $A$ has characteristic polynomial
\[  (\lambda_1 - \lambda)\prod_{i=1}^r(w_{k_i}(g(0)- 1) - \lambda), \]
from which we can read off the eigenvalues stated in the lemma. 
\end{proof}

We now calculate the right eigenvector of $A$ associated with $\lambda_1$. Let $v_{1,1} = f(k_1)/(\lambda_1 + w_{k_1}(1 - g(0)))$, and define recursively for $i=2,\ldots,r$
\begin{equation}\label{eq:v1i}
 v_{1,i} = \frac{1}{\lambda_1 + w_{k_i}(1 - g(0))}\left( f(k_i) + \sum_{j=1}^{i-1}w_{k_j}g(k_i - k_j)v_{1,j}\right). 
 \end{equation}
Finally, define 
\begin{equation}\label{eq:vr+1}
v_{1,r+1} = 1 - \sum_{j=1}^r w_{k_j}v_{1,j}
\end{equation}
and 
\begin{equation}\label{eq:v}
v_1 = (v_{1,1}, v_{1,2}, \ldots, v_{1,r}, v_{1, r+1})'.
\end{equation} 

\begin{lemma}\label{lem:eigenvector}
Let $v_1$ be the vector defined above, and let $a = (w_{k_1}, w_{k_2}, \ldots, w_{k_r}, 1)'$. The vector $v_1$ is the unique right eigenvector of $A$ associated with $\lambda_1$ for which $a\cdot v_1 = 1$. 
\end{lemma}

\begin{proof}
We verify that $v_1$ is a right eigenvector of $A$ associated with $\lambda_1$. We can look instead at $A_{\lambda}'$ which is introduced in the previous proof. Since only row operations were used to get from $A - \lambda I$ to $A'_{\lambda}$, we get that $(A-\lambda I)v_1 = 0 $ if and only if $A'_{\lambda}v_1 = 0$. We therefore need only to verify that $A'_{\lambda_1} v_1 = 0$ (where all instances of $\lambda$ are replaced with $\lambda_1$). 

Along the $(r+1)$-th row of $A'_{\lambda_1}$ for any $j=1,\ldots,r$ the $j$-th entry is given by (\ref{eq:A'j}), but with $\lambda$ replaced by $\lambda_1$, which is exactly
(\ref{eq:A'rowlater}), and so is equal to 0 by the calculations performed above.
From (\ref{eq:A'r+}), the $(r+1)$-th entry in the $(r+1)$-th row is simply $\lambda_1 - \lambda_1 = 0$. Therefore, the last row of $A'_{\lambda_1}$ is all zeros and the $(r+1)$-th  entry of the vector $A'_{\lambda_1}v_1$ is 0. 

The top $r\times (r+1)$ submatrix of $A'_{\lambda_1}$ is the same as the top $r \times (r+1)$ submatrix of $A - \lambda_1 I$. After rearranging the equality (\ref{eq:v1i}) as 
\begin{equation}\label{eq:newv1i}
f(k_i) + \sum_{j=1}^{i-1}w_{k_j}g(k_i - k_j)v_{1,j} = v_{1,i}(\lambda_1 + w_{k_i}(1 - g(0)))
\end{equation}
and recalling the entries $a_{ij}$ of $A$ from (\ref{eq:A}), we see that for $i=1, \ldots, r$, the $i$-th entry of the vector $A'_{\lambda_1}v_1$ is 
\begin{align*}
&\sum_{j=1}^{i-1}a_{ij}v_{1,j} + (a_{ii}- \lambda_1)v_{1,i} + \sum_{j=i+1}^r a_{ij}v_{1,j} + a_{i,j+1}v_{1,r+1}\\
&\hspace{2mm} = f(k_i)\left(\sum_{j=1}^{r}w_{k_j}v_{1,j} + v_{1,r+1}\right) + \sum_{j=1}^{i-1}w_{k_j}g(k_i - k_j)v_{1,j} + v_{1,i}(w_{k_i}(g(0) - 1) - \lambda_1) \\
&\hspace{2mm} = f(k_i) +  \sum_{j=1}^{i-1}w_{k_j}g(k_i - k_j)v_{1,j} + v_{1,i}(w_{k_i}(g(0) - 1) - \lambda_1) \tag*{(by (\ref{eq:vr+1}))} \\
&\hspace{2mm} = v_{1,i}(\lambda_1 + w_{k_i}(1 - g(0))) + v_{1,i}(w_{k_i}(g(0) - 1) - \lambda_1) =  0. \tag*{(by (\ref{eq:newv1i}))} 
\end{align*}
Therefore $A'_{\lambda_1}v_1 = 0$. Furthermore,
\[a \cdot v_1 = \sum_{j=1}^r w_{k_j}v_{1,j} + \left(1 - \sum_{j=1}^r w_{k_j}v_{1,j}\right) = 1.\]
Since $\lambda_1$ has algebraic (and geometric) multiplicity 1, then $v_1$ is the unique vector satisfying the statement of the lemma. 
\end{proof}

\subsection{Proofs of main results}\label{sec:mainproofs}

Recall the definitions of $f(k)$ from (\ref{eq:fhook}) and (\ref{eq:fbip}), and $g(k)$ from (\ref{eq:ghook}) and (\ref{eq:gbip}) for a set of blocks $\mathcal{C}$. Recall also that $w_k = \chi k + \rho$. Let $k_1 <  \cdots < k_r$ be the first $r$ essential (out)degrees for hooking networks or bipolar networks grown from $\mathcal{C}$.

We now prove Theorem \ref{thm:hooking}; the multivariate normal limit law for the degrees of hooking networks. Our main results for bipolar networks can be proved in a very similar manner, and we only outline the differences in the proofs. 

\begin{proof}[Proof of Theorem \ref{thm:hooking}]
We look at two cases: when a block is attached to a latch that is not the master hook of the network with degree less than or equal to $k_r$, and when a block is attached to a latch of degree greater than $k_r$ or to the master hook of the network. Recall that the master hook of the network is represented by balls of special type in the urn. \\

\noindent{\sc Case I:} Let $k_j \leq k_r$ be an essential degree and suppose that at some step in the growth of the network a vertex $v$ is chosen as a latch where $\deg(v) = k_j$ and $v$ is not the master hook of the network. Suppose a block is attached to $v$. This corresponds to choosing a ball of type $k_j$. Let $k_i \leq k_r$ be an essential degree. Other than the latch, the expected number of new vertices of degree $k_i$ added to the network is equal to $f(k_i)$. If $k_i > k_j$, the probability that the degree of $v$ is increased to $k_i$ is equal to the probability of choosing a block whose hook has degree $k_i-k_j$, which is exactly $g(k_i-k_j)$. For $k_i, k_j \leq k_r$ and with $\mathbb{E}(\xi_{k_j,k_i})$ being the expected change in the number of balls of type $k_i$ in the networks when a ball of type $k_j$ is chosen, the arguments above show that 
\[ \mathbb{E}(\xi_{k_j,k_i}) = \begin{cases}
f(k_i) & i<j \\
f(k_i) - 1 & i=j \\
f(k_i) + g(k_i-k_j) & i > j.
\end{cases}\]
For every $k$ that is an essential degree greater than $k_r$, balls of special type are added instead of balls of type $k$. By a similar argument as above, the expected number of new balls of special type added corresponding to vertices of degree $k$ when a latch of degree $k_j$ is chosen is $w_{k}(f(k) + g(k-k_j))$. Summing over all essential degrees greater than $k_r$, the expected number of balls of special type added when a ball of type $k_j$ is chosen is 
\[ \mathbb{E}(\xi_{k_j,\ast}) = \sum_{k>k_r}w_k(f(k) + g(k-k_j)).\]

\noindent{\sc Case II:} Now suppose at some step the latch $v$ is either the master hook of the network or that $\deg(v) > k_r$. In either case this corresponds to choosing a ball of special type in our urn; recall that the master hook is represented by balls of special type. Suppose that a block is attached to $v$. For an arbitrary essential degree $k_i \leq k_r$, the expected number of new vertices added with degree $k_i$ is $f(k_i)$. Therefore with $\mathbb{E}(\xi_{\ast, k_i})$ being the expected number of balls of type $k_i$ added when a ball of special type is chosen, 
\[ \mathbb{E}(\xi_{\ast, k_i}) = f(k_i).\]
For any $k\geq 1$, the probability that the degree of $v$ is increased by $k$ is $g(k)$. In this case, the ball of special type is placed back in the urn along with $\chi k$ new balls of special type. For any $k > k_r$, the expected number of new vertices with degree $k$ is once again $f(k)$. Therefore, summing over all values of $k$, the expected change in the number of balls of special type in the urn is 
\[ \mathbb{E}(\xi_{\ast,\ast}) =\sum_{k>k_r}w_k f(k) + \sum_{k \geq 1} \chi kg(k).\]

Let $\mathbb{E}(\xi_{k_j}) :=  (\mathbb{E}\xi_{k_j,k_1}, \ldots, \mathbb{E}\xi_{k_j,k_r}, \mathbb{E}\xi_{k_j, \ast})$ for $j=1 ,\ldots, r$ and for the special type~$\ast$ let $\mathbb{E}(\xi_{\ast}) := (\mathbb{E}\xi_{\ast, k_1}, \ldots, \mathbb{E}\xi_{\ast,k_r}, \mathbb{E}\xi_{\ast,\ast})$. The activity of each ball of type $k_j \leq k_r$ is $w_{k_j}$, and the activity of the ball of special type $\ast$ is 1. 
The intensity matrix is therefore the matrix $A$ whose columns are $w_{k_j}\mathbb{E}(\xi_{k_j})$ for $j=1, \ldots, r$ and whose $(r+1)$-th column is $\mathbb{E}(\xi_{\ast})$. This is precisely the matrix given in (\ref{eq:A}),  
with $g(0) = 0$. 

By Lemma \ref{lem:eigenvalues}, we get that the intensity matrix $A$ has largest real eigenvalue $\lambda_1 = \sum_{k \geq 1}(w_kf(k) + \chi k g(k)) > 0$, and the other eigenvalues $-w_{k_1}, -w_{k_2}, \ldots, -w_{k_r}$ are all negative (and so less than $\lambda_1/2$). 

The vector $v_1$ defined in (\ref{eq:v}) with $g(0) = 0$ and restricted to the first $r$ entries is exactly the vector $\nu$ defined in (\ref{eq:vhook}). Theorem \ref{thm:hooking} now follows immediately from Lemma \ref{lem:eigenvector}, and Theorem \ref{thm:urns}.
\end{proof}

\begin{proof}[Proof of Corollary \ref{thm:hookingspecial}]
Every time a new block $G_i$ with hook $h_i$ is attached to the hooking network by fusing $h_i$ with the latch $v$, any new vertex $u$ of $G_i$ added to the network is represented either by a ball of type $\deg(u)$ (with activity $\chi \deg(u) + \rho$) or by $\chi \deg(u) + \rho$ balls of special type (with activity 1). As for the latch $v$, one of the following cases applies:
\begin{itemize}
\item a ball of activity $\chi\deg(v) + \rho$ is removed and replaced with a ball of activity $\chi(\deg(v) + \deg(h_i)) + \rho$, 
\item a ball of activity $\chi\deg(v) + \rho$ is removed and replaced with $\chi(\deg(v) + \deg(h_i)) + \rho$ balls of special type (with activity 1), or
\item an additional $\chi\deg(h_i)$ balls of special type are added.
\end{itemize}
In any case the change in the total activity of the urn is
\[ s_i = \chi\deg(h_i) + \sum_{u \in V(G_i)\setminus\{h_i\}}( \chi \deg(u) + \rho) = 2\chi|E(G_i)| + \rho(|V(G_i)| - 1),\]
where the last equality holds thanks to the handshaking lemma (the sum of the degrees in a graph is twice the number of edges). Suppose that all $s_i$ are equal. The change in total activity is equal at every step, independent of which block is attached. Therefore, the corresponding urn is balanced. By \cite[Remark 1.9]{JAPO:18}, the urn satisfies the conditions of \cite[Theorem 1.1]{JAPO:18}, and so by Remark \ref{rem:moms}, Corollary \ref{thm:hookingspecial} holds. 
\end{proof} 

Theorem \ref{thm:bipolar} and Corollary \ref{thm:bipolarspecial} are proved in a similar manner to the two proofs above. We therefore omit the details, and only specify where the proofs differ. 

\begin{proof}[Proof of Theorem \ref{thm:bipolar}:]
The probability that the degree of a latch $v$ is increased by $k$ is now the probability of choosing a block whose north pole had outdegree $k + 1$ (since an arc is removed from $v$ when a block is attached). This probability is exactly defined to be $g(k)$. If a north pole has outdegree 1, then the outdegree of $v$ is not changed, and so the probability that the outdegree of $v$ is unchanged is $g(0)$. With similar arguments as in the proof of Theorem \ref{thm:hooking}, we can calculate the intensity matrix. The only differences between the intensity matrix for bipolar networks and that for hooking networks are the first $r$ diagonal entries, which are 
\[ w_{k_i}\mathbb{E}(\xi_{k_i,k_i}) = w_{k_i}(f(k_i) - g(0) - 1)\]
for $i=1, \ldots, r$ in the case of bipolar networks. The value $\mathbb{E}(\xi_{\ast,\ast})$ is the same as before since $\chi k g(k) = 0$ when $k=0$.

From Lemma \ref{lem:eigenvalues}, we get that the largest real eigenvalue of the intensity matrix is $\lambda_1 = \sum_{k \geq 1}(w_kf(k) + \chi k g(k)) > 0$ and the other eigenvalues are 
\[ w_{k_1}(g(0) - 1), w_{k_2}(g(0) - 1), \ldots, w_{k_r}(g(0) - 1).\]
Since $g(0) \leq 1$, each eigenvalue $\lambda \neq \lambda_1$ is non-positive, and so is less than $\lambda_1/2$. The vector $v_1$ defined in (\ref{eq:v}) restricted to the first $r$ entries is exactly the vector $\psi$ defined in (\ref{eq:vbipolar}), and the result now follows just as in the proof of Theorem \ref{thm:hooking}. 
\end{proof}

\begin{proof}[Proof of Corollary \ref{thm:bipolarspecial}]
Since an arc is removed at each step, the total change in activity when block $B_i$ is attached is (by similar argument to the proof of Corollary \ref{thm:hookingspecial})
\[ s_i = \chi\deg^+(N_i) - \chi +  \!\!\!\!\!\!\!\! \sum_{u \in V(B_i)\setminus\{N_i, S_i\}}\!\!\!\!\!\!\!\!\!\!\!(\chi\deg^+(u) + \rho) = \chi(|E(B_i)| - 1) + \rho(|V(B_i)| - 1).\]
If all the $s_i$'s are equal for every block, then once again the urn is balanced and Corollary \ref{thm:bipolarspecial} holds by \cite[Theorem 1.1]{JAPO:18} and Remark \ref{rem:moms}.
\end{proof}

\begin{remark}\label{rem:naughtagain}
From Remark \ref{rem:naught} we know that the initial configuration of our urn does not effect the limiting distribution. This means that we may let the original block used to make $\mathcal{G}_0$ or $\mathcal{B}_0$ to be chosen at random, or to be deterministic. It also means that if we wanted to change the probability of choosing the master hook of a hooking network or the master source of a bipolar network, we can simply change the number of balls of special type at the beginning of the urn process.
\end{remark}

\begin{remark}\label{rem:cov}
We can say something more about the covariance matrices $\Sigma$ of Theorems \ref{thm:hooking} and \ref{thm:bipolar}.  With the activity vector $a = (w_{k_1}, \ldots, w_{k_r}, 1)$, 
\[a\cdot \mathbb{E}(\xi_{\ast}) = \sum_{k \geq 1} (w_kf(k) + \chi k g(k)) = \lambda_1,\]
and for $j=1, \ldots, r$, 
\begin{align*}
a\cdot \mathbb{E}(\xi_{k_j}) &= \sum_{k\geq1} w_kf(k) + \sum_{i\geq j}w_{k_i}g(k_i - k_j) - w_{k_j} \\
&=  \sum_{k\geq1} w_kf(k) + \sum_{i\geq j}(w_{k_j} + \chi(k_i - k_j))g(k_i - k_j) - w_{k_j} \\
&= \sum_{k \geq 1}w_kf(k) + w_{k_j}\sum_{k \geq 0}g(k) + \sum_{k\geq 0}\chi k g(k) - w_{k_j} \\
&= \sum_{k \geq 1}w_kf(k) + \sum_{k\geq 0}\chi kg(k) = \lambda_1,
\end{align*}
with the last line following from property (G1) of Proposition \ref{lem:ass}. Thus, we see that Theorem \ref{thm:urns} (ii) applies with $c = \lambda_1$ and $\Sigma = \lambda_1\Sigma_1$, where $\Sigma_1$ is defined in (\ref{eq:sigma1}). 
\end{remark}

\begin{remark}\label{rem:covdiag}
Furthermore, if $\chi > 0$, then the values $w_k = \chi k+ \rho$ are all different, and so from Lemma \ref{lem:eigenvalues}, all of the eigenvalues of $A$ are different. In this case, the matrix $A$ is diagonalizable, and so Theorem \ref{thm:urns} (iii) applies and $\Sigma$ can be calculated from (\ref{eq:sigmadiag}). The diagonalizability of $A$ does not hold in general, see for example the matrix $A$ of (\ref{eq:Abip}).
\end{remark}

\section*{Acknowledgment}
We would like to acknowledge and thank Hosam Mahmoud for introducing us to the topic of hooking networks.
We would also like to thank the referees for their useful comments.






\end{document}